\documentclass[dvipdfmx]{amsart}

\usepackage{amssymb}
\usepackage{amsfonts}
\usepackage{mathrsfs}
\usepackage{bbm}
\usepackage{tikz}
\theoremstyle{plain}
 \newtheorem{theorem}{Theorem}[section]
 \newtheorem{lemma}[theorem]{Lemma}
 \newtheorem{proposition}[theorem]{Proposition}

\theoremstyle{definition}
 \newtheorem{definition}{Definition}[section]
 \newtheorem{remark}{Remark}[section]
 \newtheorem*{acknowledgement}{Acknowledgement}
\theoremstyle{remark}
 
\numberwithin{equation}{section}

\renewcommand{\l}{\left}
\renewcommand{\r}{\right}
\newcommand{\cleq}{\lesssim}

\newcommand{\ceq}{\approx} %\sim, \eqsim, \simeq, or \approx

\newcommand{\eps}{\varepsilon}
\newcommand{\R}{{\mathbb R}}
\newcommand{\C}{{\mathbb C}}
\newcommand{\Z}{{\mathbb Z}}
\newcommand{\N}{{\mathbb N}}
\def\norm#1{\left\|#1 \right\|} %norm
\def\jbra#1{\left\langle #1\right\rangle} %japanese bracket
\def\tbra#1#2{\left\langle #1 , #2\right\rangle} %inner product% trianguler bracket
\newcommand{\cA}{\mathcal{A}}
\newcommand{\cD}{\mathcal{D}}
\newcommand{\cF}{\mathcal{F}}
\newcommand{\cG}{\mathcal{G}}
\newcommand{\cN}{\mathcal{N}}

\newcommand{\cW}{\mathcal{W}}
\newcommand{\scB}{\mathscr{B}}
\newcommand{\scG}{\mathscr{G}}
\DeclareMathOperator{\re}{Re}
\begin{document}

\title[Strichartz's estimates and behavior of sol.s for critical NLDW]{The Strichartz estimates for the damped wave equation and the behavior of solutions for the energy critical nonlinear equation}

\author[T. Inui]{Takahisa Inui}
\address{Department of Mathematics, Graduate School of Science, Osaka University, Toyonaka, Osaka 560-0043, Japan}
\email{inui@math.sci.osaka-u.ac.jp}
\date{}

\maketitle

\begin{abstract}
For the linear damped wave equation (DW), the $L^p$-$L^q$ type estimates have been well studied. Recently, Watanabe \cite{Wat17} showed the Strichartz estimates for DW when $d=2,3$. In the present paper, we give Strichartz estimates for DW in higher dimensions. Moreover, by applying the estimates, we give the local well-posedness of the energy critical nonlinear damped wave equation (NLDW) $\partial_t^2 u - \Delta u +\partial_t u = |u|^{\frac{4}{d-2}}u$, $(t,x) \in [0,T) \times \mathbb{R}^d$, where $3 \leq d \leq 5$. Especially, we show the small data global existence for NLDW. 
In addition, we investigate the behavior of the solutions to NLDW. Namely, we give a decay result for solutions with finite Strichartz norm and a blow-up result for solutions with negative Nehari functional. 
\end{abstract}

\tableofcontents

\section{Introduction}

\subsection{Backgroud}

We consider the damped wave equation. 
\begin{align}
\label{DW}
\left\{
\begin{array}{ll}
	\partial_t^2 \phi - \Delta \phi +\partial_t \phi = 0, & (t,x) \in (0,\infty) \times \mathbb{R}^d,
	\\
	(\phi(0),\partial_t \phi(0)) = (\phi_0,\phi_1), & x\in \mathbb{R}^d,
\end{array}
\right.
\end{align}
where $d \in \mathbb{N}$, $(\phi_0,\phi_1)$ is given, and $\phi$ is an unknown complex valued function. 

Matsumura \cite{Mat76} applied the Fourier transform to \eqref{DW} and obtained the formula
\begin{align*}
	\phi(t,x)=\cD(t) (\phi_0+\phi_1) +\partial_t \cD(t) \phi_0,
\end{align*}
where $\cD(t)$ is defined by
\begin{align*}
	\cD(t):=e^{-\frac{t}{2}}\cF^{-1} L(t,\xi) \cF
\end{align*} 
with 
\begin{align*}
	L(t,\xi):=
	\l\{
	\begin{array}{ll}
	\displaystyle
	\frac{\sinh(t \sqrt{1/4-|\xi|^2})}{\sqrt{1/4-|\xi|^2}} & \text{if } |\xi|<1/2,
	\\
	\ 
	\\
	\displaystyle
	\frac{\sin(t \sqrt{|\xi|^2-1/4})}{\sqrt{|\xi|^2-1/4}} & \text{if } |\xi|>1/2.
	\end{array}
	\r.
\end{align*}
By this formula, Matsumura \cite{Mat76} proved the $L^p$-$L^q$ type estimate:
\begin{align}
\label{eq1.2}
	\norm{\phi(t)}_{L^p}
	\cleq \jbra{t}^{-\frac{d}{2}\l( \frac{1}{q} - \frac{1}{p}\r)}  \norm{(\phi_0,\phi_1)}_{L^q\times L^q} 
	+e^{-\frac{t}{4}} \l( \norm{\phi_0}_{H^{\l[\frac{d}{2}\r]+1}} + \norm{\phi_1}_{H^{\l[\frac{d}{2}\r]}} \r),
\end{align}
where $1\leq q \leq 2 \leq p \leq \infty$ and $[d/2]$ denotes the integer part of $d/2$. Such $L^p$-$L^q$ type estimates have been studied well. See \cite{Nis03,HoOg04,Nar05} and references therein. 
The $L^p$-$L^q$ type estimates for the heat equation and the wave equation are also well studied. We recall the $L^p$-$L^q$ type estimate for the heat equation $\partial_t v - \Delta v = 0$:
\begin{align*}
	\norm{\cG(t) g}_{L^p} \cleq t^{-\frac{d}{2}\l( \frac{1}{q} - \frac{1}{p}\r)} \norm{g}_{L^q},
\end{align*}
where $1\leq q \leq p \leq \infty$ and $\cG(t):=\cF^{-1} e^{-t|\xi|^2} \cF$. We also refer to the $L^p$-$L^q$ type estimate for the wave equation $\partial_t^2 w - \Delta w =0$: 
\begin{align*}
	\norm{\cW(t)g}_{L^p} \cleq |t|^{-2d\l(\frac{1}{2}-\frac{1}{p}\r)} \norm{g}_{\dot{W}^{\gamma-1,p'}},
\end{align*}
for $2\leq p <\infty$ and $(d+1)(1/2-1/p) \leq \gamma < d$, where $\cW(t):=\cF^{-1} \sin(t|\xi|) / |\xi|\cF$. See \cite{Bre75}.
Matsumura's estimate \eqref{eq1.2} shows that the solution of \eqref{DW} behaves like the solution of the heat equation and the wave equation in some sense. More precisely, the low frequency part of the solution to the damped wave equation behaves like the solution of the heat equation and the high frequency part behaves like the solution of the wave equation but decays exponentially (see \cite{IIOWp} for another $L^p$-$L^q$ estimate). 

For the heat equation and the wave equation, by using the $L^p$-$L^q$ type estimates, we obtain the  space-time estimates, what we call the Strichartz estimate. The Strichartz estimates for the heat equation are 
\begin{align*}
	\norm{v}_{L_{t}^{q} (I: L_{x}^{r}(\R^d))} 
	\cleq \norm{v_0}_{L^2} + \norm{F}_{L_{t}^{\tilde{q}'} (I: L_{x}^{\tilde{r}'}(\R^d))},
\end{align*}
where $v$ satisfies $\partial_t v - \Delta v =F$ with $v(0)=v_0$ and $2/q+d/r=2/\tilde{q}+d/\tilde{r}=d/2$. See \cite{Wei81,Gig86}. 
We also have the Strichartz estimates for the wave equation as follows. 
\begin{align*}
	\norm{w}_{L_{t}^{q} (I: L_{x}^{p}(\R^d))} 
	\cleq \norm{w_0}_{\dot{H}^1} +\norm{w_1}_{L^2} + \norm{F}_{L_{t}^{\tilde{q}'} (I: L_{x}^{\tilde{r}'}(\R^d))},
\end{align*}
where $w$ satisfies $\partial_t^2 w - \Delta w=F$ with $(w(0),\partial_t w(0))=(w_0,w_1)$ and $1/q+d/r=d/2-1=1/\tilde{q}'+d/\tilde{r}'-2$. See \cite{GiVe95}. 
In the present paper, we give the Strichartz estimates for the damped wave equation. 
Recently, Watanabe \cite{Wat17} obtained the Strichartz estimates for the damped wave equation when $d=2,3$ by an energy method.
In this paper, we give the Strichartz estimates by a duality argument for $d=2,3$ and higher dimensions.  
We also consider the energy critical nonlinear damped wave equation.
\begin{align}
\label{NLDW}
\tag{NLDW}
\left\{
\begin{array}{ll}
	\partial_t^2 u - \Delta u +\partial_t u = |u|^{\frac{4}{d-2}}u, & (t,x) \in [0,T) \times \mathbb{R}^d,
	\\
	(u(0),\partial_t u(0)) = (u_0,u_1), & x\in \mathbb{R}^d,
\end{array}
\right.
\end{align}
where $d \geq 3$, $(u_0,u_1)$ is given, and $u$ is an unknown complex valued function. 
We will show the local well-posedness for \eqref{NLDW} when $3\leq d \leq 5$ by applying the Strichartz estimates. 
The existence of a local solution has been studied by \cite{Kap94,IkIn16,IkWa17p} (see also \cite{Kap87I,Kap90II,Kap90III}).
However, the small data global existence has not been known. Using the Strichartz estimates which are proved in this paper, we can show not only the existence of a local solution but also the small data global existence for \eqref{NLDW}. 

Moreover, we discuss the global behavior of the solutions to \eqref{NLDW}. For the energy critical nonlinear heat equation, the solution with a bounded global space-time norm decays to zero (see e.g. \cite{GuRo17p}). On the other hand, there exist finite time blow-up solutions by Levine \cite{Lev73}.  For the energy critical nonlinear wave equation, the energy is conserved by the flow. There exist solutions which scatter to the solutions of the free wave equation and finite time blow-up solutions by Payne and Sattinger \cite{PaSa75}. See also \cite{KeMe08}. In the present paper, we show that the solution to \eqref{NLDW} with a finite space-time norm decays since the energy decays like the heat equation and there exist finite time blow-up solutions.

%%%%%%%%%%%%%%%%%%%%%%%%%%%%%%%%%%%%%%%%%%%%%%%%%%%%%%%%

\subsection{Main results}

We state main results. First, we obtain the Strichartz estimates for \eqref{DW}. The so-called admissible pairs can be taken as same as in the heat case since the $L^p$-$L^q$ type estimate of the low frequency part is similar to the heat estimate and the high frequency part decays exponentially in time. However, the derivative loss appears from the high frequency part which is wave-like part. 

\begin{proposition}[Homogeneous Strichartz estimates]
\label{prop1.1}
Let $d \geq 2$, $2 \leq r < \infty$, and $2\leq q \leq \infty$. Set $\gamma:= \max\{ d(1/2 - 1/r)-1/q, \frac{d+1}{2}(1/2-1/r)\}$. Assume 
\begin{align*}	\frac{d}{2}\l( \frac{1}{2} - \frac{1}{r}\r) \geq \frac{1}{q},
\end{align*}
Then, we have
\begin{align*}
	\norm{\cD(t)f}_{L_{t}^{q} (I: L_{x}^{r}(\R^d))} 
	&\cleq \norm{ \jbra{\nabla}^{\gamma-1} f}_{L^2},
	\\
	\norm{\partial_t \cD(t)f}_{L_{t}^{q} (I: L_{x}^{r}(\R^d))} 
	&\cleq \norm{ \jbra{\nabla}^{\gamma} f}_{L^2},
	\\
	\norm{\partial_t^2 \cD(t)f}_{L_{t}^{q} (I: L_{x}^{r}(\R^d))} 
	&\cleq \norm{ \jbra{\nabla}^{\gamma+1} f}_{L^2}.
\end{align*}
\end{proposition}

\begin{remark}
\label{rmk1.1}
We note that  the homogeneous Strichartz estimate holds in the heat end-point case \textit{i.e.} $(q,r)=(2, 2d/(d-2))$ when $d\geq3$. 
\end{remark}

\begin{proposition}[Inhomogeneous Strichartz estimates]
\label{prop1.2}
Let $d\geq2$, $2\leq r,\tilde{r} < \infty$, and $2\leq q, \tilde{q} \leq \infty$. We set $\gamma:= \max \{ d(1/2-1/r)-1/q, \frac{d+1}{2}(1/2-1/r) \}$ and $\tilde{\gamma}$ in the same manner. Assume that $(q,r)$ and $(\tilde{q},\tilde{r})$ satisfies 
\begin{align*}
	\frac{d}{2}\l( \frac{1}{2} - \frac{1}{r}\r) +\frac{d}{2}\l( \frac{1}{2} - \frac{1}{\tilde{r}}\r)
	>\frac{1}{q} + \frac{1}{\tilde{q}},
\end{align*}
\begin{align*}
	\frac{d}{2} \l( \frac{1}{2} - \frac{1}{r}\r) + \frac{d}{2} \l( \frac{1}{2} - \frac{1}{\tilde{r}}\r) 
	= \frac{1}{q} + \frac{1}{\tilde{q}}
	\text{ and }
	1< \tilde{q}' < q<\infty,
\end{align*}
or 
\begin{align*}
	(q,r)=(\tilde{q},\tilde{r})=(\infty,2).
\end{align*}
Moreover, we exclude the end-point case, that is, we assume $(q,r) \neq (2,2(d-1)/(d-3)))$ and $(\tilde{q},\tilde{r}) \neq (2,2(d-1)/(d-3)))$ when $d \geq 4$. 
Then, we have
\begin{align*}
	\norm{\int_{0}^{t} \cD(t-s) F(s) ds}_{L_{t}^{q} (I: L_{x}^{r}(\R^d))}
	&\cleq \norm{ \jbra{\nabla}^{\gamma+\tilde{\gamma}+\delta-1} F}_{L_{t}^{\tilde{q}'} (I: L_{x}^{\tilde{r}'}(\R^d))},
	\\
	\norm{\int_{0}^{t}  \partial_t \cD(t-s) F(s) ds}_{L_{t}^{q} (I: L_{x}^{r}(\R^d))}
	&\cleq \norm{ \jbra{\nabla}^{\gamma+\tilde{\gamma}+\delta} F}_{L_{t}^{\tilde{q}'} (I: L_{x}^{\tilde{r}'}(\R^d))},
\end{align*}
where $\delta = 0$ when $\frac{1}{\tilde{q}}(1/2-1/r)=\frac{1}{q}(1/2-1/\tilde{r})$ and in the other cases $\delta \geq 0$ is defined in the table 1 below. 
\begin{table}[htb]
\begingroup
\renewcommand{\arraystretch}{1.6}
\begin{tabular}{|c|c|c|} 
	\hline
	$\delta$
		&  $\frac{1}{\tilde{q}} \l( \frac{1}{2}- \frac{1}{r}\r) < \frac{1}{q} \l( \frac{1}{2} - \frac{1}{\tilde{r}} \r) $
		&  $\frac{1}{\tilde{q}} \l( \frac{1}{2}- \frac{1}{r}\r) > \frac{1}{q} \l( \frac{1}{2} - \frac{1}{\tilde{r}} \r)$ 
	\\[4pt]
	\hline \hline
	$\frac{d-1}{2}\l( \frac{1}{2}- \frac{1}{r}\r) \geq \frac{1}{q}$ 
	& $0$
	& $0$
	\\
	$\frac{d-1}{2}\l( \frac{1}{2}- \frac{1}{\tilde{r}}\r) \geq \frac{1}{\tilde{q}}$
	&  
	& 
	\\[5pt]
	\hline
	$\frac{d-1}{2}\l( \frac{1}{2}- \frac{1}{r}\r) \geq \frac{1}{q}$ 
	& $\times$
	& $\frac{\tilde{q}}{q}  \l\{ \frac{1}{\tilde{q}}-\frac{d-1}{2} \l( \frac{1}{2} - \frac{1}{\tilde{r}}\r) \r\}$
	\\
	$\frac{d-1}{2}\l( \frac{1}{2}- \frac{1}{\tilde{r}}\r) < \frac{1}{\tilde{q}}$
	&  
	& 
	\\[5pt]
	\hline
	$\frac{d-1}{2}\l( \frac{1}{2}- \frac{1}{r}\r) < \frac{1}{q}$ 
	& $\frac{q}{\tilde{q}}  \l\{ \frac{1}{q}-\frac{d-1}{2} \l( \frac{1}{2} - \frac{1}{r}\r) \r\}$
	& $\times$
	\\
	$\frac{d-1}{2}\l( \frac{1}{2}- \frac{1}{\tilde{r}}\r) \geq \frac{1}{\tilde{q}}$
	&  
	& 
	\\[5pt]
	\hline
	$\frac{d-1}{2}\l( \frac{1}{2}- \frac{1}{r}\r) < \frac{1}{q}$ 
	& $\frac{1}{\tilde{q}} \frac{d-1}{2} \l\{ \tilde{q}\l( \frac{1}{2}- \frac{1}{\tilde{r}}\r) - q\l( \frac{1}{2} - \frac{1}{r}\r) \r\}$
	& $\frac{1}{q} \frac{d-1}{2} \l\{ q\l( \frac{1}{2}- \frac{1}{r}\r) - \tilde{q}\l( \frac{1}{2} - \frac{1}{\tilde{r}}\r) \r\}$
	\\
	$\frac{d-1}{2}\l( \frac{1}{2}- \frac{1}{\tilde{r}}\r)< \frac{1}{\tilde{q}}$
	&
	&
	\\
	\hline
\end{tabular}
\caption{The value of $\delta$. ($\times$ means that the case does not occur.)}
\endgroup
\end{table}
\end{proposition}

\begin{remark}
If $(q,r)$ satisfies the wave admissible condition $(d-1)(1/2-1/r)/2 \geq 1/q$, then the derivative loss is same as that in the Strichartz estimates for the wave equation \textit{i.e.} $\gamma=d(1/2-1/r)-1/q$. And thus, we need more derivative if $(q,r)$ is the pair between the wave case and  the heat case, \textit{i.e.} $d(1/2-1/r)/2 \geq 1/q > (d-1)(1/2-1/r)/2$.
\end{remark}

\begin{remark}
Chen, Fang, and Zhang \cite{CFZ15} considered  the damped fractional wave equation $\partial_t^2 v +(-\Delta)^{\alpha} v + 2\partial_t v=0$, where $\alpha>0$. They claimed that the better homogeneous Strichartz estimates can be obtained, where ``better" means the derivative loss is less than that of the wave equation. However, at least when $\alpha=1$, their proof seems to be imcomplete. %(In Remark 3.3 in \cite{CFZ15}, they stated that an estimate is uniformly in a parameter. However, this is not correct when $\alpha=1$.)
%the estimate is not uniformly) I guess that the derivative loss of the usual damped wave equation is same as the wave equation. 
\end{remark}

Moreover, applying these Strichartz estimates, we can show the local well-posedness, especially small data global existence, for energy critical nonlinear damped wave equations \eqref{NLDW}.

\begin{definition}[solution]
Let $T \in (0,\infty]$. 
We say that $u$ is a solution to \eqref{NLDW} on $[0,T)$ if $u$ satisfies $(u,\partial_t u) \in C([0,T):H^1(\R^d)\times L^2(\R^d))$, $\jbra{\nabla}^{1/2} u \in  L_{t,x}^{2(d+1)/(d-1)}(I)$ and $ u \in L_{t,x}^{2(d+1)/(d-2)}(I)$ for any compact interval $I \subset [0,T)$, $(u(0),\partial_t u(0)) = (u_0,u_1)$, and the Duhamel's formula
\begin{align*}
	u(t,x)=\cD(t) (u_0+u_1) +\partial_t \cD(t) u_0 + \int_{0}^{t} \cD(t-s) ( |u(s)|^{\frac{4}{d-2}}u(s) ) ds
\end{align*}
for all $t \in [0,T)$. We say that $u$ is global if $T=\infty$. 
\end{definition}

We have the following local well-posedness result when $3 \leq d \leq 5$.
\begin{theorem}[local well-posedness]
\label{thm1.3}
Let $d\in\{3,4,5\}$ and $T\in (0,\infty]$. Let $(u_0,u_1) \in H^1(\R^d)\times L^2(\R^d)$ satisfy $\| (u_0,u_1) \|_{H^1 \times L^2} \leq A$. Then, there exists $\delta=\delta(A)>0$ such that if 
\begin{align*}
	\norm{\cD(t) (u_0+u_1) +\partial_t \cD(t) u_0}_{L_{t,x}^{\frac{2(d+1)}{d-2}}([0,T))} \leq \delta,
\end{align*}
then there exists a solution $u$ to \eqref{NLDW} with $\|  u \|_{L_{t,x}^{2(d+1)/(d-2)}}([0,T)) \leq 2 \delta$. Moreover, we have the standard blow-up criterion, that is, if the maximal existence time $T_{+}=T_{+}(u_0,u_1)$ is finite, then the solution satisfies $\norm{u}_{L^{2(d+1)/(d-2)}([0,T_{+}))}=\infty$. 
\end{theorem}

%%%%%%%%

\begin{remark}
If $\| (u_0,u_1) \|_{H^1 \times L^2} \ll 1$, by the Strichartz estimates, we can take $T=\infty$. Namely, the small data global existence holds. 
\end{remark}	

\begin{remark}
See the sequel paper \cite{InWap} for the local well-posedness and small data global existence of \eqref{NLDW} when $d \geq 6$. The difficulty of $d \geq 6$ comes from the loss of differentiability of the nonlinear term. We need to attention to the estimate of the difference. 
\end{remark}

\begin{remark}
The existence of local solution is well known (see \cite{IkIn16,IkWa17p}). However, the small data global existence has not been known except for low dimension cases. (Watanabe \cite{Wat17} showed the small data global existence when $d=3$.)
\end{remark}

\begin{remark}
As it is well known, we can obtain the local well-posedness of the nonlinear damped wave equation with the more general nonlinearity in the same way as Theorem \ref{thm1.3}. Namely, we find the local well-posedness for the following equation.
\begin{align}
\label{eq1.3}
\l\{
\begin{array}{ll}
	\partial_t^2 u - \Delta u +\partial_t u = \cN(u), & (t,x) \in (0,\infty) \times \mathbb{R}^d,
	\\
	(u(0),\partial_t u(0)) = (u_0,u_1), & x\in \mathbb{R}^d.
\end{array}
\r.
\end{align} 
Assume that the nonlinearity $\cN: \mathbb{C} \to \mathbb{C}$ is continuously differentiable and obeys the power type estimates
\begin{align*}
	\cN(z) 
	&=O(|z|^{1+\frac{4}{d-2}}),
	\\
	\cN_{z}(z), \cN_{\bar{z}}(z) 
	&=O(|z|^{\frac{4}{d-2}}),
	\\
	\cN_{z}(z)-\cN_{z}(w), \cN_{\bar{z}}(z)-\cN_{\bar{z}}(w)
	&=O(|z-w|^{\min\{1,\frac{4}{d-2}\}}(|z|+|w|)^{\max\{0,\frac{6-d}{d-2}\}} ),
\end{align*}
where $\cN_{z}$ and $\cN_{\bar{z}}$ are the usual derivatives 
\begin{align*}
	\cN_{z}:= \frac{1}{2} \l( \frac{\partial \cN}{\partial x} - i\frac{\partial \cN}{\partial y} \r), 
	\quad
	\cN_{z}:= \frac{1}{2} \l( \frac{\partial \cN}{\partial x} + i\frac{\partial \cN}{\partial y} \r)
\end{align*}
for $z=x+iy$. 
The typical examples are $\cN(u)=\lambda |u|^{1+4/(d-2)}$ or $\lambda |u|^{4/(d-2)}u$ with $\lambda \in \C\setminus\{0\}$. 
\end{remark}

We have the energy $E$ of \eqref{NLDW}, which is defined by
\begin{align*}
	E(u,\partial_t u) = \frac{1}{2} \norm{\nabla u}_{L^2}^2 + \frac{1}{2} \norm{\partial_t u}_{L^2}^2 -\frac{d-2}{2d} \norm{u}_{L^{\frac{2d}{d-2}}}^{\frac{2d}{d-2}}.
\end{align*}
If $u$ is a solution to \eqref{NLDW}, then the energy satisfies 
\begin{align*}
	\frac{d}{dt} E(u(t),\partial_t u(t)) = -\norm{\partial_t u(t)}_{L^2}^2
\end{align*}
for all $t \in (0,T_{\max})$. This means the energy decay. This observation shows us that some global solutions may decay. Indeed,  we can prove that a global solution with a finite Strichartz norm decays to $0$ in the energy space as follows. 

\begin{theorem}
\label{thm1.4}
If $u$ is a global solution of \eqref{NLDW} with $\| u \|_{L_{t,x}^{2(d+1)/(d-2)}([0,\infty))}<\infty$, then $u$ satisfies
\begin{align*}
	\lim_{t \to \infty}  \l(\norm{u(t)}_{H^1} + \norm{\partial_t u(t)}_{L^2}\r) =0.
\end{align*}
\end{theorem}

\begin{remark}
This is similar to the energy critical nonlinear heat equation. See Gustafson and Roxanas \cite{GuRo17p}.
\end{remark}

\begin{remark}
Theorem \ref{thm1.4} holds for all dimensions $d \geq 3$ since we need to treat the estimate of the difference unlike the local well-posedness. 
\end{remark}

At last, we discuss the blow-up of the solutions to \eqref{NLDW}. We set 
\begin{align*}
	J(\varphi)
	&:=\frac{1}{2} \norm{\nabla \varphi}_{L^2}^2 - \frac{d-2}{2d} \norm{\varphi}_{L^{\frac{2d}{d-2}}}^{\frac{2d}{d-2}},
	\\
	K(\varphi)
	&:= \norm{ \nabla \varphi}_{L^2}^2-\norm{\varphi}_{L^{\frac{2d}{d-2}}}^{\frac{2d}{d-2}}. 
\end{align*}
Then, it is well known that the minimizing problem
\begin{align*}
	\mu := \inf \l\{ J(\varphi): \varphi \in \dot{H}^1 \setminus \{0\}, K(\varphi)=0 \r\}
\end{align*}
is attained by the Talenti function
\begin{align*}
	W(x):= \l\{ 1+ \frac{|x|^2}{d(d-2)}\r\}^{-\frac{d-2}{2}}.
\end{align*}
See \cite{Tal76}. 
Here, the Talenti function satisfies the following nonlinear elliptic equation
\begin{align*}
	-\Delta W = |W|^{\frac{4}{d-2}}W, \quad x \in \R^d. 
\end{align*}
Therefore, $W$ is a static solution to \eqref{NLDW}. Then, we get the following blow-up result. 

\begin{theorem}
\label{thm1.5}
Let $(u_0,u_1) \in H^1(\R^d)\times L^2(\R^d)$ belong to 
\begin{align*}
	\scB:=\{(u_0,u_1) \in H^1(\R^d)\times L^2(\R^d) : E(u_0,u_1) < \mu, K(u_0)<0\}. 
\end{align*}
Then the solution to \eqref{NLDW} blows up in finite time. 
\end{theorem}

\begin{remark}
The theorem means that the static solution $W$ is strongly unstable for \eqref{NLDW}. 
\end{remark}

\begin{remark}
The proof of Theorem \ref{thm1.5} is essentially given by Ohta \cite{Oht97}. He showed the blow-up result for abstract setting by the method of an ordinary differential inequality instead of by the so-called concavity argument, which is well applied to wave or Klein-Gordon equation. 
\end{remark}

We collect some notations. For the exponent $p$, we denote the H\"{o}lder conjugate of $p$ by $p'$. The bracket $\jbra{\cdot}$ is Japanese bracket \textit{i.e.} $\jbra{a}:=(1+|a|^2)^{1/2}$. 

We use $A \lesssim B$ to denote the estimate $A \leq CB$ with some constant $C>0$.
The notation $A \sim B$ stands for $A \lesssim B$ and $A \lesssim B$.

Let $\chi_{\leq1} \in C_0^{\infty} (\mathbb{R})$ be a cut-off function satisfying $\chi_{\leq1} (r)=1$ for $|r| \le 1$ and $\chi_{\leq1}(r) =0$ for $|r| \ge 2$ and let $\chi_{>1}=1-\chi_{\leq 1}$. 

For a function $f : \mathbb{R}^n \to \mathbb{C}$,
we define the Fourier transform and the inverse Fourier transform by
\begin{align*}
	\mathcal{F} [f] (\xi) = \hat{f} (\xi) = (2\pi)^{-n/2}
		\int_{\mathbb{R}^n} e^{-i x \xi} f(x) \,dx,\quad
	\mathcal{F}^{-1} [f] (x) = (2\pi)^{-n/2}
		\int_{\mathbb{R}^n} e^{i x \xi} f(\xi) \,dx.
\end{align*}

For a measurable function $m = m(\xi)$, we denote
the Fourier multiplier $m(\nabla)$ by
\begin{align*}
	m(\nabla) f (x) = \cF^{-1} \left[ m(\xi) \hat{f}(\xi) \right] (x).
\end{align*}

For $s \in \mathbb{R}$ and $1\leq p \leq \infty$,
we denote the usual Sobolev space by
\begin{align*}
	W^{s,p}(\mathbb{R}^d) := \left\{ f \in \mathcal{S}' (\mathbb{R}^d) :
		\| f \|_{W^{s,p}} = \| \langle \nabla \rangle^s f \|_{L^p} < \infty \right\}.
\end{align*}
We write $H^s(\mathbb{R}^d) := W^{s,2}(\mathbb{R}^d)$ for simplicity. Let $\dot{W}^{s,p}(\mathbb{R}^d)$ and $\dot{H}^{s}(\R^d)$ denote the corresponding homogeneous Sobolev spaces. 

We define $P_{\leq1}:= \cF^{-1} \chi_{\leq 1} \cF$, $P_{>1}:=\cF^{-1} \chi_{> 1} \cF$, and 
\begin{align*}
	P_{N} = \cF^{-1} \l( \chi_{\leq1}\l( \frac{\xi}{N}\r) -  \chi_{\leq1}\l( \frac{2\xi}{N}\r) \r)\cF
\end{align*}
for $N \in 2^{\Z}$. 
For a time interval $I$ and $F:I\times \R^d \to \C$, we set 
\begin{align*}
	\norm{F}_{L^{q}(I:L^r)(\R^d)}:=  \l(\int_{I} \norm{F(t,\cdot)}_{L^r(\R^d)}^q dt \r)^{1/q}
\end{align*}
and $L_{t,x}^{q}(I):=L^{q}(I:L^{q}(\R^d))$. We sometimes use $L_s^p$ and $L_t^p$ to uncover time variables $s$ and $t$.

This paper is structured as follows. 
Section \ref{sec2} is devoted to show the Strichartz estimates. In particular, we give the Strichartz estimates for low frequency part in Section \ref{sec2.1} and that for high frequency part  in Section \ref{sec2.2}. In Setion \ref{sec3}, we prove the locall well-posedness of \eqref{NLDW} by the Strichartz estimates. In Section \ref{sec4}, we discuss the decay of the global solutions to \eqref{NLDW} with a bounded space-time norm. Section \ref{sec5} contains the proof of the blow-up result.

%%%%%%%%%%%%%%%%%%%%%%%%%%%%%%%%%%%%%%%%%%%%%%%%%%%%%%%
%%%%%%%%%%%%%%%%%%%%%%%%%%%%%%%%%%%%%%%%%%%%%%%%%%%%%%%

%%%%%%%%%%%%%%%%%%%%%%%%%%%%%%%%%%%%%%%%%%%%%%%%%%%%%%%%
%%%%%%%%%%%%%%%%%%%%%%%%%%%%%%%%%%%%%%%%%%%%%%%%%%%%%%%%

\section{The Strichartz estimates}
\label{sec2}

We split $\cD$ to low frequency part $\cD_{l}$ and high frequency part $\cD_{h}$ as follows. 
\begin{align*}
	\cD_{l}(t)&:=\cD(t)P_{\leq1},
	\\
	\cD_{h}(t)&:=\cD(t)P_{>1}. 
\end{align*}
In this section, we prove the Strichartz estimates for low and high frequency parts respectively.

%%%%%%%%%%%%%%%%%%%%%%%%%%%%%%%%%%%%%%%%%%%%%%%%%%%%%%%%

\subsection{The Strichartz estimates for low frequency part}
\label{sec2.1}

We have the $L^p$-$L^q$ type estimates from \cite{CFZ14} and \cite{IIOWp}. These estimates are similar to those of the heat equation. 

\begin{lemma}[$L^{r}$-$L^{\tilde{r}}$ estimate for low frequency part \cite{CFZ14,IIOWp}]
\label{lem2.1}
Let $1 \leq \tilde{r} \leq r \leq \infty$ and $\sigma \geq 0$. Then, we have
\begin{align*}
	\norm{|\nabla|^{\sigma} \cD_{l}(t) f}_{L^r} 
	\cleq \jbra{t}^{-\frac{d}{2}\l( \frac{1}{\tilde{r}} - \frac{1}{r}\r)-\frac{\sigma}{2}} \norm{f}_{L^{\tilde{r}}}, 
\end{align*}
for any $t > 0$ and $f \in L^{\tilde{r}}(\R^d)$. We also have
\begin{align*}
	\norm{|\nabla|^{\sigma}  \partial_t \cD_{l}(t) f}_{L^r} 
	&\cleq \jbra{t}^{-\frac{d}{2}\l( \frac{1}{\tilde{r}} - \frac{1}{r}\r)-\frac{\sigma}{2}-1} \norm{f}_{L^{\tilde{r}}},
	\\
	\norm{|\nabla|^{\sigma}  \partial_t^2 \cD_{l}(t) f}_{L^r} 
	&\cleq \jbra{t}^{-\frac{d}{2}\l( \frac{1}{\tilde{r}} - \frac{1}{r}\r)-\frac{\sigma}{2}-2} \norm{f}_{L^{\tilde{r}}}.
\end{align*}
\end{lemma}

By these $L^p$-$L^q$ type estimates, we obtain the following homogeneous Strichartz estimate.

\begin{lemma}[Homogeneous Strichartz estimate for low frequency part]
\label{lem2.2}
Let $\sigma\geq 0$. Let $1 \leq \tilde{r} \leq r \leq \infty$ and  $1\leq q \leq \infty$. 
Assume that they satisfy 
\begin{align*}
	\frac{d}{2} \l( \frac{1}{\tilde{r}} - \frac{1}{r}\r)  > \frac{1}{q},
\end{align*}
or
\begin{align*}
	\frac{d}{2} \l( \frac{1}{\tilde{r}} - \frac{1}{r}\r)  = \frac{1}{q}
	\text{ and }
	q >\tilde{r}>1.
\end{align*}
Then, for any  $f \in L^{\tilde{r}}(\R^d)$, 
\begin{align*}
	\norm{ \jbra{\nabla}^{\sigma} \cD_{l}(t) f}_{L^{q}(I:L^r(\R^d))} 
	\cleq  \norm{f}_{L^{\tilde{r}}}, 
\end{align*}
where $I \subset [0,\infty)$ is a time interval and the implicit constant is independent of $I$. Moreover, we also have
\begin{align*}
	\norm{ \jbra{\nabla}^{\sigma} \partial_t \cD_{l}(t) f}_{L^{q}(I:L^r(\R^d))} 
	&\cleq  \norm{f}_{L^{\tilde{r}}},
	\\
	\norm{ \jbra{\nabla}^{\sigma} \partial_t^2 \cD_{l}(t) f}_{L^{q}(I:L^r(\R^d))} 
	&\cleq  \norm{f}_{L^{\tilde{r}}}.
\end{align*}
\end{lemma}

\begin{proof}
These Strichartz estimates are same as those of the heat equation. Thus, the same proof does work. 
However, we give the proof for reader's convenience.

We first consider the case of $\frac{d}{2}(1/\tilde{r}-1/r)>1/q$. By the $L^r$-$L^{\tilde{r}}$ estimate (Lemma \ref{lem2.1}), 
\begin{align*}
	\norm{ \jbra{\nabla}^{\sigma}\cD_{l}(t) f}_{L^r}
	&\cleq \norm{\cD_{l}(t) f}_{L^r}+\norm{|\nabla|^{\sigma}\cD_{l}(t) f}_{L^r}
	\\
	&\cleq \jbra{t}^{-\frac{d}{2}\l( \frac{1}{\tilde{r}} - \frac{1}{r}\r)} \norm{f}_{L^{\tilde{r}}}
	+ \jbra{t}^{-\frac{d}{2}\l( \frac{1}{\tilde{r}} - \frac{1}{r}\r)-\frac{\sigma}{2}} \norm{f}_{L^{\tilde{r}}}
	\\
	&\cleq \jbra{t}^{-\frac{d}{2}\l( \frac{1}{\tilde{r}} - \frac{1}{r}\r)} \norm{f}_{L^{\tilde{r}}}.
\end{align*}
Then, we obtain
\begin{align*}
	\norm{ \jbra{\nabla}^{\sigma}\cD_{l}(t) f}_{L^{q}(I:L^r(\R^d))} 
	\cleq \norm{\jbra{t}^{-\frac{d}{2} \l( \frac{1}{\tilde{r}} - \frac{1}{r}\r) } \norm{f}_{L^{\tilde{r}}}}_{L^q([0,\infty))}
	\cleq  \norm{f}_{L^{\tilde{r}}}.
\end{align*}
Next, we consider the second case. We set $Tf:= \norm{ \jbra{\nabla}^{\sigma}\cD_{l}(t) f}_{L^r(\R^d)}$ and $(q_{1},r_{1}):=(\infty,r)$ and $(q_{2},r_{2})=(\rho,\gamma)$, where $(\rho,\gamma)$ satisfies $\frac{d}{2} \l( 1/\gamma - 1/r\r)  = 1/\rho$ and $\rho,\gamma>1$. Then $T$ is sub-additive and we have $T:L^{r_{j}}(\R^d) \to L^{q_{j},\infty}([0,\infty))$ for $j=1,2$. Indeed, we have
\begin{align*}
	\norm{Tf(t)}_{L^\infty(I)} 
	&\cleq \norm{\jbra{\nabla}^{\sigma} \cD_{l}(t) f}_{L^{\infty}(I:L^r(\R^d))} 
	\cleq \norm{f}_{L^r},
	\\
	\norm{Tf(t)}_{L^{\rho,\infty}(I)} 
	&\cleq \norm{\jbra{\nabla}^{\sigma} \cD_{l}(t) f}_{L^{\rho,\infty}(I:L^r(\R^d))} 
	\cleq \norm{f}_{L^{\gamma}}.
\end{align*}
If $\rho \geq \gamma$, we can use the Marcinkiewicz interpolation theorem so that we have
\begin{align*}
	\norm{\cD_{l}(t) f}_{L^{q}(I:L^r(\R^d))} 
	\cleq  \norm{f}_{L^{\tilde{r}}}, 
\end{align*}
for $(q,\tilde{r})$ satisfying $q>\tilde{r}>1$ and
\begin{align*}
	\frac{1}{q} = \frac{1-\theta}{q_1} + \frac{\theta}{q_2}, \quad \frac{1}{\tilde{r}} = \frac{1-\theta}{r_1} + \frac{\theta}{r_2}, \quad 0< \theta <1. 
\end{align*}
This means that the desired inequality holds for $(q,r)$ such that $\frac{d}{2}(1/\tilde{r}-1/r)=1/q$ and $q>\tilde{r}>1$. See also \cite{Wei81,Gig86}. In the same way, we get the second and the third inequalities. 
\end{proof}

\begin{remark}
As stated in Remark \ref{rmk1.1},  the Strichartz estimate in the heat end-point case $(q,r)=(2,2d/(d-2))$ holds for $d \geq 3$. However, we exclude the end-point case in Lemma \ref{lem2.2} since it is not clear whether the end-point Strichartz estimate holds or not for $q=\tilde{r}$ and $\tilde{r}\neq 2$. We give the proof of the Strichartz estimate in the end-point case $(q,r)=(2,2d/(d-2))$ for $d \geq 3$ in Section \ref{sec2.3} (see Lemma \ref{lem2.11}). 
\end{remark}

\begin{lemma}[Inhomogeneous Strichartz estimate for low frequency part]
\label{lem2.3}
Let $\sigma\geq 0$. Let $1 \leq \tilde{r}' \leq r \leq \infty$ and  $1\leq q, \tilde{q} \leq \infty$. 
Assume that they satisfy 
\begin{align*}
	\frac{d}{2}\l( \frac{1}{2} - \frac{1}{r}\r) +\frac{d}{2}\l( \frac{1}{2} - \frac{1}{\tilde{r}}\r)
	>\frac{1}{q} + \frac{1}{\tilde{q}},
\end{align*}
\begin{align*}
	\frac{d}{2} \l( \frac{1}{2} - \frac{1}{r}\r) + \frac{d}{2} \l( \frac{1}{2} - \frac{1}{\tilde{r}}\r) 
	= \frac{1}{q} + \frac{1}{\tilde{q}}
	\text{ and }
	1< \tilde{q}' < q<\infty,
\end{align*}
or 
\begin{align*}
	(q,r)=(\tilde{q},\tilde{r})=(\infty,2).
\end{align*}
Then it holds that 
\begin{align*}
	\norm{ \jbra{\nabla}^{\sigma} \int_{0}^{t} \cD_{l}(t-s) F(s) ds}_{L^{q}(I:L^r(\R^d))}
	&\cleq \norm{F}_{L^{\tilde{q}'}(I:L^{\tilde{r}'}(\R^d))},
	\\
	\norm{\int_{0}^{t} \partial_t \cD_{l}(t-s) F(s) ds}_{L^{q}(I:L^r(\R^d))}
	&\cleq \norm{F}_{L^{\tilde{q}'}(I:L^{\tilde{r}'}(\R^d))},
\end{align*}
where $I \subset [0,\infty)$ is a time interval such that $0 \in \overline{I}$ and the implicit constant is independent of $I$. 
\end{lemma}

\begin{proof}
We only show the first estimate since the second can be proved similarly. 
Applying the $L^r$-$L^{\tilde{r}}$ estimate (Lemma \ref{lem2.1}), we obtain
\begin{align*}
	\norm{\jbra{\nabla}^{\sigma} \int_{0}^{t} \cD_{l}(t-s) F(s) ds}_{L^{q}(I:L^r(\R^d))}
	&\cleq \norm{\int_{0}^{t} \norm{\jbra{\nabla}^{\sigma} \cD_{l}(t-s) F(s)}_{L^r} ds}_{L^{q}(I)}
	\\
	& \cleq \norm{\int_{0}^{t} \jbra{t-s}^{-\frac{d}{2}\l( \frac{1}{\tilde{r}'} - \frac{1}{r}\r)} \norm{F(s)}_{L^{\tilde{r}'}} ds}_{L^{q}(I)}.
\end{align*}
When $\frac{d}{2} \l( \frac{1}{2} - \frac{1}{r}\r) + \frac{d}{2} \l( \frac{1}{2} - \frac{1}{\tilde{r}}\r)> \frac{1}{q} + \frac{1}{\tilde{q}}$, by the Young inequality, we obtain
\begin{align*}
	 \norm{\int_{0}^{t} \jbra{t-s}^{-\frac{d}{2}\l( \frac{1}{\tilde{r}'} - \frac{1}{r}\r)} \norm{F(s)}_{L^{\tilde{r}'}} ds}_{L^{q}(I)}
	 &\cleq \norm{\jbra{t}^{-\frac{d}{2}\l( \frac{1}{\tilde{r}'} - \frac{1}{r}\r)}}_{L^{\frac{q\tilde{q}}{\tilde{q}+q}}} \norm{F}_{L^{\tilde{q}'}(I:L^{\tilde{r}'}(\R^d))}
	 \\
	 &\cleq \norm{F}_{L^{\tilde{q}'}(I:L^{\tilde{r}'}(\R^d))}.
\end{align*}
On the other hand, when $\frac{d}{2} \l( \frac{1}{2} - \frac{1}{r}\r) + \frac{d}{2} \l( \frac{1}{2} - \frac{1}{\tilde{r}}\r)= \frac{1}{q} + \frac{1}{\tilde{q}}$ and $1<\tilde{q}'<q<\infty$, applying the Hardy--Littlewood--Sobolev inequality, we obtain
\begin{align*}
	 \norm{\int_{0}^{t} \jbra{t-s}^{-\frac{d}{2}\l( \frac{1}{\tilde{r}'} - \frac{1}{r}\r)} \norm{F(s)}_{L^{\tilde{r}'}} ds}_{L^{q}(I)}
	 \cleq \norm{F}_{L^{\tilde{q}'}(I:L^{\tilde{r}'}(\R^d))}.
\end{align*}
When $(q,r)=(\tilde{q},\tilde{r})=(\infty,2)$, the inequality is trivial. 
This completes the proof. 
\end{proof}

%%%%%%%%%%%%%%%%%%%%%%%%%%%%%%%%%%%%%%%%%%%%%%%%%%%%%%%%

\subsection{The Strichartz estimates for high frequency part}
\label{sec2.2}

Since we have
\begin{align*}
	\cD_{h}(t) = e^{-\frac{t}{2}} \cF^{-1} \frac{e^{it \sqrt{|\xi|^2-1/4}} - e^{-it \sqrt{|\xi|^2-1/4}}}{2i \sqrt{|\xi|^2-1/4}} \chi_{>1} (\xi) \cF,
\end{align*}
it is enough to estimate
\begin{align*}
	e^{-t/2} e^{\pm it\sqrt{-\Delta-1/4}} P_{>1}.
\end{align*}

\begin{lemma}[Homogeneous Strichartz estimate  for high frequency part]
\label{lem2.4}
Let $d\geq 2$. Let $2\leq r < \infty$ and $2 \leq q \leq \infty$. We set $\gamma:= \max \{ d(1/2-1/r)-1/q, \frac{d+1}{2}(1/2-1/r) \}$. Then, we have 
\begin{align*}
	\norm{e^{-t/2} e^{\pm it\sqrt{-\Delta-1/4}} P_{>1} f}_{L^{q}(I:L^r(\R^d))}
	\cleq \norm{|\nabla|^{\gamma} f}_{L^2}
\end{align*}
where $I \subset [0,\infty)$ is a time interval and the implicit constant is independent of $I$. 
In particular, we have
\begin{align*}
	\norm{\cD_{h}(t) f}_{L^{q}(I:L^r(\R^d))}
	&\cleq \norm{|\nabla|^{\gamma} \jbra{\nabla}^{-1} f}_{L^2},
	\\
	\norm{ \partial_t \cD_{h}(t) f}_{L^{q}(I:L^r(\R^d))}
	&\cleq \norm{|\nabla|^{\gamma} f}_{L^2},
	\\
	\norm{ \partial_t^2 \cD_{h}(t) f}_{L^{q}(I:L^r(\R^d))}
	&\cleq \norm{|\nabla|^{\gamma+1} f}_{L^2}.
\end{align*}
\end{lemma}

\begin{proof}
First, we consider $e^{it\sqrt{-\Delta-1/4}}$.
We note that
\begin{align*}
	e^{it \sqrt{-\Delta-1/4}} = e^{it|\nabla|} e^{it \left( \sqrt{-\Delta-1/4}-|\nabla| \right)}.
\end{align*}
Since we have
\begin{align*}
	\l| \sqrt{|\xi|^2-1/4}-|\xi| \r| 
	= \frac{1}{4(\sqrt{|\xi|^2-1/4}+|\xi|)}
	\ceq |\xi|^{-1},
\end{align*}
a simple calculation shows
\begin{align*}
	\l| \partial_{\xi}^{\alpha} e^{it \l( \sqrt{|\xi|^2-1/4}-|\xi| \r)} \r|
	\cleq  \jbra{t}^{|\alpha|} |\xi|^{-|\alpha|}
\end{align*}
for $\xi \neq 0$ and $\alpha \in \mathbb{Z}_{\ge 0}^n$. 
Thus, the Mihlin--H\"{o}rmander multiplier theorem (see \cite[Theorem 6.2.7]{Gra14}) gives
\begin{align*}
	\norm{ e^{ it\sqrt{-\Delta-1/4}} P_{>1} f}_{L^r}
	\cleq  \jbra{t}^{\delta_r} \norm{ e^{ it |\nabla|} f }_{L^r}
\end{align*}
for some $\delta_r>0$. Therefore, we obtain
\begin{align*}
	\norm{e^{-t/2} e^{ it\sqrt{-\Delta-1/4}} P_{>1} f}_{L^{q}(I:L^r(\R^d))}
	&= \norm{e^{-t/2} \norm{ e^{ it\sqrt{-\Delta-1/4}} P_{>1} f }_{L^r}  }_{L^{q}(I)}
	\\
	&\cleq \norm{e^{-t/2} \jbra{t}^{\delta_r} \norm{ e^{ it |\nabla|} f }_{L^r}  }_{L^{q}(I)}
	\\
	&\cleq \norm{ e^{ it |\nabla|} f  }_{L^{\tilde{q}}(I:L^r(\R^d))},
\end{align*}
where we have used the H\"{o}lder inequality in the last inequality and we take $\tilde{q}$ such that
\begin{align*}
	\tilde{q}
	=
	\l\{
	\begin{array}{ll}
	q 
	& \text{ if } \frac{d-1}{2}\l( \frac{1}{2} - \frac{1}{r}\r) \geq \frac{1}{q}, 
	\\
	 \l\{ \frac{d-1}{2} \l( \frac{1}{2} - \frac{1}{r} \r)  \r\}^{-1}  
	& \text{ if } \frac{d-1}{2}\l( \frac{1}{2} - \frac{1}{r}\r) < \frac{1}{q}.
	\end{array}
	\r.
\end{align*}
Then, $(\tilde{q},r)$ is a wave admissible pair. Namely, it satisfies
\begin{align*}
	\frac{1}{\tilde{q}} + \frac{d-1}{2r} \leq \frac{d-1}{4}, 
	\ \tilde{q},r,d \geq 2,
	\text{ and } (q,r,d) \neq (2,\infty,3)
\end{align*}
and
\begin{align*}
	\frac{1}{\tilde{q}}+\frac{d}{r}=\frac{d}{2} - \gamma,
\end{align*}
where we note that $\gamma \geq 0$. 
Therefore, by the Strichartz estimate for the free wave equation (see \cite{GiVe95} or \cite[Corollary 2.5 in p.233]{KTV14}), we get 
\begin{align*}
	\norm{e^{-t/2} e^{ it\sqrt{-\Delta-1/4}} P_{>1} f}_{L^{q}(I:L^r(\R^d))}
	&\cleq \norm{ e^{ it |\nabla|} f  }_{L^{\tilde{q}}(I:L^r(\R^d))}
	\\
	& \cleq \norm{|\nabla|^{\gamma} f}_{L^2}.
\end{align*}
Similarly, we also have
\begin{align*}
	\norm{e^{-t/2} e^{- it\sqrt{-\Delta-1/4}} P_{>1} f}_{L^{q}(I:L^r(\R^d))}
	 \cleq \norm{|\nabla|^{\gamma} f}_{L^2}. 
\end{align*}
Combining them with the formula of $\cD_{h}$, we obtain
\begin{align*}
	\norm{\cD_{h}(t) f}_{L^{q}(I:L^r(\R^d))}
	\cleq \norm{|\nabla|^{\gamma}\jbra{\nabla}^{-1} f}_{L^2},
\end{align*}
where we use $\sqrt{|\xi|^2-1/4} \ceq \jbra{\xi}$ for $|\xi|\geq1$. 
Moreover, we also get the estimates related to $\partial_t \cD_{h}(t)$ and $\partial_t^2 \cD_{h}(t)$. 
\end{proof}

\begin{remark}
We can also obtain the homogeneous Strichartz estimates for high frequency part when $1\leq q < 2$. Indeed, taking
\begin{align*}
	\tilde{q}
	=
	\l\{
	\begin{array}{ll}
	2 
	& \text{ if } \frac{d-1}{2}\l( \frac{1}{2} - \frac{1}{r}\r) \geq \frac{1}{2},
	\\
	&
	\\
	\l\{ \frac{d-1}{2} \l( \frac{1}{2} - \frac{1}{r} \r)  \r\}^{-1}  
	& \text{ if } \frac{d-1}{2}\l( \frac{1}{2} - \frac{1}{r}\r) < \frac{1}{2},
	\end{array}
	\r.
\end{align*}
 $(\tilde{q},r)$ is a wave admissible pair and thus the above argument does work. We note that, in this case, we need to redefine $\gamma$ such that 
 \begin{align*}
 	\gamma:= \max\l\{ \frac{d+1}{2}\l(\frac{1}{2}-\frac{1}{r}\r), d\l(\frac{1}{2}-\frac{1}{r}\r)-\frac{1}{2}\r\} \geq 0.
 \end{align*}
\end{remark}

%%%%%%%%%%%%%%%%%%%%%%%%%%%%%%%%%%%%%%%%%%%%%%%%%%%%%%%

To prove inhomogeneous Strichartz estimates for high frequency part, we show the $L^p$-$L^{q}$ type estimate.

\begin{lemma}[$L^{r}$-$L^{r'}$ estimate for high frequency part]
\label{lem2.5}
Let $d \geq 1$. Let $2 \leq r < \infty$. Then, it holds that 
\begin{align*}
	\norm{e^{\pm it \sqrt{-\Delta-1/4} }P_{>1} P_{N} f}_{L^{r}} 
	\cleq  \jbra{t}^{\delta_r}  (1+|t|N)^{-\frac{d-1}{2} \l( 1-\frac{2}{r}\r) } N^{d\l( 1 -\frac{2}{r} \r)} \norm{P_{N} f}_{L^{r'}}
\end{align*}
for any $t>0$ and $N \in 2^{\Z}$, where $\delta_r$ is a positive constant. 
\end{lemma}

\begin{proof}
%We have
%\begin{align*}
%	\norm{e^{\pm it \sqrt{-\Delta-1/4} }P_{>1} P_{N} f}_{L^{2}} 
%	\leq \norm{P_{N} f}_{L^{2}}
%\end{align*}
%so that it is enough by the interpolation to show that
%\begin{align*}
%	\norm{e^{\pm it \sqrt{-\Delta-1/4} }P_{>1} P_{N} f}_{L^{\infty}} 
%	\cleq (1+|t|N)^{-\frac{d-1}{2} } N^{d} \norm{P_{N} f}_{L^{1}}. 
%\end{align*}
Combining the $L^{p}$-$L^{q}$ type estimate for free wave equation (see \cite{Bre75} or \cite[Lemma 2.1 in p.230]{KTV14}) and the Mihlin--H\"{o}rmander multiplier theorem, we get the statement. 
\end{proof}

%%%%%%%%%%%%%%%%%%%%%%%%%%%%%%%%%%%%%%%%%%%%%%%%%%%%%%%%

\begin{lemma}[Inhomogeneous Strichartz estimate for high frequency part]
\label{lem2.6}
Let $d \geq 2$. Let $2 \leq r < \infty$ and $2 \leq q \leq \infty$. We set $\gamma:= \max \{ d(1/2-1/r)-1/q, \frac{d+1}{2}(1/2-1/r) \}$. We exclude the end-point case, that is, we assume that $(q,r) \neq (2, 2(d-1)/(d-3))$ when $d \geq 4$. Then, we have 
\begin{align*}
	\norm{\int_{0}^{t} e^{-\frac{t-s}{2}} e^{\pm i(t-s)\sqrt{-\Delta-1/4}} P_{>1} P_{N} F(s) ds}_{L^{q}(I:L^r(\R^d))}
	\cleq N^{2\gamma} \norm{P_{N}F}_{L^{q'}(I:L^{r'}(\R^d))},
\end{align*}
where $I \subset [0,\infty)$ is a time interval such that $0 \in \overline{I}$ and the implicit constant is independent of $I$. 
\end{lemma}

\begin{proof}
By the $L^{r}$-$L^{r'}$ estimate for high frequency part, Lemma \ref{lem2.5}, we get
\begin{align}
\label{eq2.1}
	&\norm{\int_{0}^{t} e^{-\frac{t-s}{2}} e^{\pm i(t-s)\sqrt{-\Delta-1/4}} P_{>1} P_{N} F(s) ds}_{L^{q}(I:L^r(\R^d))}
	\\ \notag
	&\leq \norm{\int_{0}^{t} e^{-\frac{t-s}{2}}  \jbra{t-s}^{\delta_r}  (1+|t-s|N)^{-\frac{d-1}{2} \l( 1-\frac{2}{r}\r) } N^{d\l( 1 -\frac{2}{r} \r)} \norm{P_{N} F(s)}_{L^{r'}}  ds}_{L^{q}(I)}
	\\ \notag
	&\cleq 
	N^{d\l( 1 -\frac{2}{r} \r)} 
	\norm{\int_{0}^{t} e^{-\frac{t-s}{4}} (1+|t-s|N)^{-\frac{d-1}{2} \l( 1-\frac{2}{r}\r) }  \norm{P_{N} F(s)}_{L^{r'}}  ds}_{L^{q}(I)}.
\end{align}
Here, by the Young inequality, we obtain
\begin{align}
\label{eq2.2}
	&N^{d\l( 1 -\frac{2}{r} \r)} 
	\norm{\int_{0}^{t} e^{-\frac{t-s}{4}} (1+|t-s|N)^{-\frac{d-1}{2} \l( 1-\frac{2}{r}\r) }  \norm{P_{N} F(s)}_{L^{r'}}  ds}_{L^{q}(I)}
	\\ \notag
	& \cleq N^{d\l( 1 -\frac{2}{r} \r)} 
	\norm{ e^{-\frac{\cdot}{4}} (1+|\cdot|N)^{-\frac{d-1}{2} \l( 1-\frac{2}{r}\r) } }_{L^{q/2}([0,\infty))}
	\norm{P_N F}_{L^{q'}(I: L^{r'}(\R^d))}.
\end{align}
In the case of $\frac{d-1}{2} (1-2/r)> 2/q$, since we have
\begin{align*}
	\norm{ e^{-\frac{\cdot}{4}} (1+|\cdot|N)^{-\frac{d-1}{2} \l( 1-\frac{2}{r}\r) } }_{L^{q/2}([0,\infty))}^{q/2}
	\leq \int_{0}^{\infty} (1+|t|N)^{-\frac{d-1}{2} \l( 1-\frac{2}{r}\r) \frac{q}{2} } dt
	\cleq N^{-1},
\end{align*}
we obtain, from \eqref{eq2.1} and \eqref{eq2.2}, 
\begin{align*}
	\text{(L.H.S. of \eqref{eq2.1})}
	& \cleq N^{2 \l\{ d\l( \frac{1}{2} - \frac{1}{r} \r) - \frac{1}{q} \r\}} \norm{P_N F}_{L^{q'}(I: L^{r'}(\R^d))}
	\\
	& = N^{2\gamma} \norm{P_N F}_{L^{q'}(I: L^{r'}(\R^d))}. 
\end{align*}
On the other hand, in the case of $\frac{d-1}{2} (1-2/r) <  2/q$, we have
\begin{align*}
	&\norm{ e^{-\frac{\cdot}{4}} (1+|\cdot|N)^{-\frac{d-1}{2} \l( 1-\frac{2}{r}\r) } }_{L^{q/2}([0,\infty))}^{q/2}
	\\
	&= \int_{0}^{\infty} e^{-\frac{q}{8}t} (1+|t|N)^{-\frac{d-1}{2} \l( 1-\frac{2}{r}\r) \frac{q}{2} } dt
	\\
	&\leq N^{-\frac{d-1}{2} \l( 1-\frac{2}{r}\r) \frac{q}{2}} \int_{0}^{\infty} e^{-\frac{q}{8}t} t^{-\frac{d-1}{2} \l( 1-\frac{2}{r}\r) \frac{q}{2} } dt
	\\
	&\leq N^{-\frac{d-1}{2} \l( 1-\frac{2}{r}\r) \frac{q}{2}}
	 \l( \int_{0}^{1} t^{-\frac{d-1}{2} \l( 1-\frac{2}{r}\r) \frac{q}{2} } dt
	+ \int_{1}^{\infty} e^{-\frac{q}{8}t}dt \r)
	\\
	& \cleq N^{-\frac{d-1}{2} \l( 1-\frac{2}{r}\r) \frac{q}{2}}.
\end{align*}
Therefore, we obtain, from \eqref{eq2.1} and \eqref{eq2.2}, 
\begin{align*}
	\text{(L.H.S. of \eqref{eq2.1})}
	& \cleq N^{2 \l\{  \frac{d+1}{2} \l( \frac{1}{2}-\frac{1}{r}\r) \r\}} \norm{P_N F}_{L^{q'}(I: L^{r'}(\R^d))}
	\\
	& = N^{2\gamma} \norm{P_N F}_{L^{q'}(I: L^{r'}(\R^d))}. 
\end{align*}
At last, we consider the case of $\frac{d-1}{2} (1-2/r) =  2/q$. Then, we have
\begin{align}
\label{eq2.3}
	&N^{d\l( 1 -\frac{2}{r} \r)} 
	\norm{\int_{0}^{t} e^{-\frac{t-s}{4}} (1+|t-s|N)^{-\frac{d-1}{2} \l( 1-\frac{2}{r}\r) }  \norm{P_{N} F(s)}_{L^{r'}}  ds}_{L^{q}(I)}
	\\ \notag
	& \cleq N^{d\l( 1 -\frac{2}{r} \r)}  N^{-\frac{d-1}{2} \l( 1-\frac{2}{r}\r) } 
	\norm{\int_{0}^{t} |t-s|^{-\frac{d-1}{2} \l( 1-\frac{2}{r}\r) }  \norm{P_{N} F(s)}_{L^{r'}}  ds}_{L^{q}(I)}
	\\ \notag
	& = N^{2\gamma} 
	\norm{\int_{0}^{t} |t-s|^{-\frac{d-1}{2} \l( 1-\frac{2}{r}\r) }  \norm{P_{N} F(s)}_{L^{r'}}  ds}_{L^{q}(I)}
\end{align}
and it follows from the Hardy--Littlewood--Sobolev inequality that 
\begin{align}
\label{eq2.4}
	\norm{\int_{0}^{t} |t-s|^{-\frac{d-1}{2} \l( 1-\frac{2}{r}\r) }  \norm{P_{N} F(s)}_{L^{r'}}  ds}_{L^{q}(I)}
	\cleq  \norm{P_{N} F(s)}_{L^{q'}(I:L^{r'}(\R^d))},
\end{align}
since $(q,r)$ is not the end-point. 
Combining \eqref{eq2.1}, \eqref{eq2.3}, and \eqref{eq2.4}, we get the desired inequality. 
\end{proof}

\begin{remark}
In the previous lemma, we exclude the end-point case. However, we can easily obtain the Strichartz estimate in the end-point case if we permit additional derivative loss. Indeed, using \eqref{eq2.1} and the following calculation:
\begin{align*}
	&N^{d\l( 1 -\frac{2}{r} \r)} 
	\norm{\int_{0}^{t} e^{-\frac{t-s}{4}} (1+|t-s|N)^{-\frac{d-1}{2} \l( 1-\frac{2}{r}\r) }  \norm{P_{N} F(s)}_{L^{r'}}  ds}_{L^{q}(I)}
	\\
	& \cleq N^{d\l( 1 -\frac{2}{r} \r)}  
	\norm{\int_{0}^{t} e^{-\frac{t-s}{4}} \norm{P_{N} F(s)}_{L^{r'}}  ds}_{L^{q}(I)}
	\\
	& \cleq N^{2d\l( \frac{1}{2} -\frac{1}{r} \r)} 
	\norm{ e^{-\frac{\cdot}{4}} }_{L^{q/2}([0,\infty))}
	\norm{P_N F}_{L^{q'}(I: L^{r'}(\R^d))}
	\\
	& \cleq N^{2d\l( \frac{1}{2} -\frac{1}{r} \r)} 
	\norm{P_N F}_{L^{q'}(I: L^{r'}(\R^d))},
\end{align*}
where we have used the Young inequality, we get the Strichartz estimate with the derivative loss $d(1/2-1/r)$ instead of $\gamma$. We note that the derivative loss $d(1/2 - 1/r)$ is larger than $\gamma$ in Lemma \ref{lem2.6}. 
\end{remark}

\begin{lemma}[$L_{t}^{\infty}L_{x}^2$-$L_{t}^{q'}L_{x}^{r'}$ estimate for high frequency part]
\label{lem2.7}
Let $d \geq2$. 
Let $2 \leq r < \infty$ and $2 \leq q \leq \infty$. We set $\gamma:= \max \{ d(1/2-1/r)-1/q, \frac{d+1}{2}(1/2-1/r) \}$. We assume that $(q,r) \neq (2, 2(d-1)/(d-3))$ when $d \geq 4$. Then, we have 
\begin{align*}
	\norm{\int_{0}^{t} e^{-\frac{t-s}{2}} e^{\pm i(t-s)\sqrt{-\Delta-1/4}} P_{>1} P_{N} F(s) ds}_{L^{\infty}(I:L^{2}(\R^d))}
	\cleq N^{\gamma} \norm{P_{N}F}_{L^{q'}(I:L^{r'}(\R^d))},
\end{align*}
where $I \subset [0,\infty)$ is a time interval such that $0 \in \overline{I}$ and the implicit constant is independent of $I$. 
\end{lemma}

\begin{proof}
Now, we have
\begin{align*}
	&\norm{\int_{0}^{t} e^{-\frac{t-s}{2}} e^{\pm i(t-s)\sqrt{-\Delta-1/4}} P_{>1} P_{N} F(s) ds}_{L^{2}}^2
	\\
	&= \tbra{\int_{0}^{t} e^{-\frac{t-s}{2}} e^{\pm i(t-s)\sqrt{-\Delta-1/4}} P_{>1} P_{N} F(s) ds}{\int_{0}^{t} e^{-\frac{t-\tau}{2}} e^{\pm i(t-\tau)\sqrt{-\Delta-1/4}} P_{>1} P_{N} F(\tau) d\tau}_{L^2}
	\\
	&=\int_{0}^{t} \int_{0}^{s}  \tbra{e^{-\frac{t-s}{2}} e^{\pm i(t-s)\sqrt{-\Delta-1/4}} P_{>1} P_{N} F(s) }{ e^{-\frac{t-\tau}{2}} e^{\pm i(t-\tau)\sqrt{-\Delta-1/4}} P_{>1} P_{N} F(\tau)}_{L^2} d\tau ds
	\\
	&\quad +\int_{0}^{t} \int_{0}^{\tau}  \tbra{e^{-\frac{t-s}{2}} e^{\pm i(t-s)\sqrt{-\Delta-1/4}} P_{>1} P_{N} F(s) }{ e^{-\frac{t-\tau}{2}} e^{\pm i(t-\tau)\sqrt{-\Delta-1/4}} P_{>1} P_{N} F(\tau)}_{L^2} ds d\tau
	\\
	&=I+I\!\!I.
\end{align*}
By the symmetry, it is enough to estimate $I$. By the H\"{o}lder inequality, $e^{-\frac{t-s}{2}}e^{-\frac{t-\tau}{2}}=e^{-(t-s)}e^{-\frac{s-\tau}{2}}$, and $e^{-(t-s)}\leq 1$ for $s \in [0,t]$ we obtain
\begin{align*}
	I
	&=\int_{0}^{t} \tbra{e^{-\frac{t-s}{2}} e^{\pm i(t-s)\sqrt{-\Delta-1/4}} P_{>1} P_{N} F(s) }{ \int_{0}^{s}  e^{-\frac{t-\tau}{2}} e^{\pm i(t-\tau)\sqrt{-\Delta-1/4}} P_{>1} P_{N} F(\tau) d\tau}_{L^2}  ds
	\\
	&\leq  \int_{0}^{t} e^{-(t-s)} \tbra{ \l| P_{N} F(s) \r|}{ \l| \int_{0}^{s} e^{-\frac{s-\tau}{2}} e^{\pm i(s-\tau)\sqrt{-\Delta-1/4}} P_{>1}^2 P_{N} F(\tau) d\tau \r|}_{L^2} ds
	\\
	&\leq \norm{P_{N} F}_{L^{q'}(I:L^{r'}(\R^d))} 
	\norm{\int_{0}^{s} e^{-\frac{s-\tau}{2}} e^{\pm i(s-\tau)\sqrt{-\Delta-1/4}} P_{>1}^2 P_{N} F(\tau) d\tau}_{L_{s}^{q}((0,t):L^{r}(\R^d))}.
\end{align*}
By Lemma \ref{lem2.6}, we obtain 
\begin{align*}
	I \leq N^{2\gamma} \norm{ P_{N} F}_{L^{q'}(I:L^{r'}(\R^d))}^2.
\end{align*}
Thus, it follows that
\begin{align*}
	\norm{\int_{0}^{t} e^{-\frac{t-s}{2}} e^{\pm i(t-s)\sqrt{-\Delta-1/4}} P_{>1} P_{N} F(s) ds}_{L^{2}}^2
	\cleq N^{2\gamma} \norm{ P_{N} F}_{L^{q'}(I:L^{r'}(\R^d))}^2.
\end{align*}
This finishes the proof. 
\end{proof}

\begin{remark}
\label{rmk2.3}
Let $T>0$, $2 \leq r < \infty$ and $2\leq q \leq \infty$, $\gamma:= \max \{ d(1/2-1/r)-1/q, \frac{d+1}{2}(1/2-1/r) \}$, and $(q,r) \neq (2, 2(d-1)/(d-3))$ when $d \geq 4$. 
Then, we  have the following inequality by the same argument as in Lemma \ref{lem2.6}. 
\begin{align}
\label{eq2.5}
	&\norm{\int_{s}^{t} e^{-(\tau-s)} e^{-\frac{t-\tau}{2}} e^{\mp i(t-\tau)\sqrt{-\Delta-1/4}} P_{>1} P_{N} F(\tau)  d\tau}_{L_{t}^{q}((s,T):L^r(\R^d))}
	\\ \notag
	&\cleq N^{2\gamma} \norm{P_{N}F}_{L^{q'}(I:L^{r'}(\R^d))},
\end{align}
where $s<T$ is a parameter. Moreover, we also have the following estimate from \eqref{eq2.5} and the similar argument to Lemma \ref{lem2.7}.
\begin{align}
\label{eq2.6}
	&\norm{\int_{s}^{T} e^{-\frac{t-s}{2}} e^{\pm i(t-s)\sqrt{-\Delta-1/4}} P_{>1} P_{N} F(t) dt}_{L_{s}^{\infty}([0,T):L^{2}(\R^d))}
	\\ \notag
	&\cleq N^{\gamma} \norm{ P_{N} F}_{L^{q'}([0,T):L^{r'}(\R^d))}.
\end{align}

\end{remark}

\begin{lemma}[$L_{t}^{q}L_{x}^{r}$-$L_{t}^{1}L_{x}^{2}$ estimate for high frequency part]
\label{lem2.8}
Let $2 \leq r < \infty$ and $2 \leq q \leq \infty$. We set $\gamma:= \max \{ d(1/2-1/r)-1/q, \frac{d+1}{2}(1/2-1/r) \}$. We assume that $(q,r) \neq (2, 2(d-1)/(d-3))$ when $d \geq 4$. Then, we have 
\begin{align*}
	\norm{\int_{0}^{t} e^{-\frac{t-s}{2}} e^{\pm i(t-s)\sqrt{-\Delta-1/4}} P_{>1} P_{N} F(s) ds}_{L^{q}(I:L^{r}(\R^d))}
	\cleq N^{\gamma} \norm{P_{N}F}_{L^{1}(I:L^{2}(\R^d))},
\end{align*}
where $I \subset [0,\infty)$ is a time interval such that $0 \in \overline{I}$ and the implicit constant is independent of $I$. 
\end{lemma}

\begin{proof}
We may write $I=[0,T)$. 
We use a standard duality argument. Let $G \in C_{0}^{\infty}(I \times \R^d)$ and $\tilde{P}_{N}:= P_{N/2}+P_{N} + P_{2N}$. Since we have $\tilde{P}_{N}P_{N}=P_{N}$, it follows from the Fubini theorem and H\"{o}lder inequality that
\begin{align}
\label{eq2.7}
	&\int_{0}^{T} \tbra{\int_{0}^{t} e^{-\frac{t-s}{2}} e^{\pm i(t-s)\sqrt{-\Delta-1/4}} P_{>1} P_{N} F(s) ds}{G(t)}dt
	\\ \notag
	&=\int_{0}^{T} \int_{0}^{t} e^{-\frac{t-s}{2}} \tbra{e^{\pm i(t-s)\sqrt{-\Delta-1/4}} P_{>1} P_{N} F(s)}{G(t)} dsdt
	\\ \notag
	&=\int_{0}^{T} \int_{s}^{T} e^{-\frac{t-s}{2}} \tbra{P_{N} F(s)}{ e^{\mp i(t-s)\sqrt{-\Delta-1/4}} P_{>1} \tilde{P}_{N} G(t)} dt ds
	\\ \notag
	&=\int_{0}^{T} \tbra{P_{N} F(s)}{  \int_{s}^{T} e^{-\frac{t-s}{2}} e^{\mp i(t-s)\sqrt{-\Delta-1/4}} P_{>1} \tilde{P}_{N} G(t) dt} ds
	\\ \notag
	& \leq \norm{P_{N}F}_{L^{1}(I:L^{2}(\R^d))} \norm{ \int_{s}^{T} e^{-\frac{t-s}{2}} e^{\mp i(t-s)\sqrt{-\Delta-1/4}} P_{>1} \tilde{P}_{N} G(t) dt}_{L^{\infty}(I:L^{2}(\R^d))}
\end{align}
By \eqref{eq2.6} in Remark \ref{rmk2.3}, we get 
\begin{align}
\label{eq2.8}
	 &\norm{ \int_{s}^{T} e^{-\frac{t-s}{2}} e^{\mp i(t-s)\sqrt{-\Delta-1/4}} P_{>1} \tilde{P}_{N} G(t) dt}_{L^{\infty}(I:L^{2}(\R^d))}
	 \\ \notag
	 &\leq \sum_{j=N/2,N,2N}  \norm{ \int_{s}^{T} e^{-\frac{t-s}{2}} e^{\mp i(t-s)\sqrt{-\Delta-1/4}} P_{>1} P_{j} G(t) dt}_{L^{\infty}(I:L^{2}(\R^d))}
	 \\ \notag
	 & \cleq N^{\gamma}  \sum_{j=N/2,N,2N} \norm{P_{j}G}_{L^{q'}(I:L^{r'}(\R^d))}
	\\ \notag
	&\cleq N^{\gamma}  \norm{G}_{L^{q'}(I:L^{r'}(\R^d))}.
\end{align}
Since we have the duality
\begin{align*}
	\norm{F}_{L^{q}(I:L^{r}(\R^d))}
	=\sup \l\{ \int_{I} \tbra{F(t)}{G(t)} dt: G\in C_{0}^{\infty}(I \times \R^d), \norm{G}_{L^{q'}(I:L^{r'}(\R^d))}=1\r\},
\end{align*}
the desired estimate follows from \eqref{eq2.7} and \eqref{eq2.8}. 
\end{proof}

Combining these estimates, we obtain the following Strichartz estimates when $(1/q,1/r)$ and $(1/\tilde{q},1/\tilde{r})$ are on a same line.

\begin{lemma}
\label{lem2.9}
Let $2 \leq r, \tilde{r} < \infty$ and $2 \leq q, \tilde{q} \leq \infty$. We set $\gamma:= \max \{ d(1/2-1/r)-1/q, \frac{d+1}{2}(1/2-1/r) \}$ and $\tilde{\gamma}$ in the same manner. Assume that
\begin{align*}
	\frac{1}{\tilde{q}} \l( \frac{1}{2}- \frac{1}{r}\r) = \frac{1}{q} \l( \frac{1}{2} - \frac{1}{\tilde{r}} \r).
\end{align*} 
We also assume that $(q,r) \neq (2, 2(d-1)/(d-3))$ and $(\tilde{q},\tilde{r}) \neq (2, 2(d-1)/(d-3))$ when $d \geq 4$. 
Then, we have 
\begin{align*}
	\norm{\int_{0}^{t} e^{-\frac{t-s}{2}} e^{\pm i(t-s)\sqrt{-\Delta-1/4}} P_{>1}  F(s) ds}_{L^{q}(I:L^{r}(\R^d))}
	\cleq  \norm{|\nabla|^{\gamma+\tilde{\gamma}}F}_{L^{\tilde{q}'}(I:L^{\tilde{r}'}(\R^d))},
\end{align*}
where $I \subset [0,\infty)$ is a time interval such that $0 \in \overline{I}$ and the implicit constant is independent of $I$. 
\end{lemma}

\begin{proof}
We set
\begin{align*}
	\Psi[F](t,x):=\int_{0}^{t} e^{-\frac{t-s}{2}} e^{\pm i(t-s)\sqrt{-\Delta-1/4}} P_{>1}  F(s) ds.
\end{align*}
First, we consider the case of $2 \leq r \leq \tilde{r}$.  Then, $\tilde{q} \leq q$ and thus there exists $\theta \in [0,1]$ such that 
\begin{align*}
	\frac{1}{q} = \frac{\theta}{\tilde{q}} + \frac{1-\theta}{\infty},
	\quad
	\frac{1}{r} = \frac{\theta}{\tilde{r}} + \frac{1-\theta}{2}.
\end{align*}
By this formula, we have $\theta \tilde{\gamma}= \gamma$. 
Therefore, by the H\"{o}lder inequality, Lemmas \ref{lem2.6} and \ref{lem2.7}, we obtain 
\begin{align*}
	&\norm{\Psi[F]}_{L^{q}(I:L^{r}(\R^d))}
	\\
	&\cleq \norm{\Psi[F]}_{L^{\tilde{q}}(I:L^{\tilde{r}}(\R^d))}^{\theta} \norm{\Psi[F]}_{L^{\infty}(I:L^{2}(\R^d))}^{1-\theta}
	\\
	&\cleq  \l(N^{2\tilde{\gamma}} \norm{P_N F}_{L^{\tilde{q}'}(I:L^{\tilde{r}'}(\R^d))} \r)^{\theta}  
	\l(N^{\tilde{\gamma}} \norm{P_N F}_{L^{\tilde{q}'}(I:L^{\tilde{r}'}(\R^d))} \r)^{1-\theta}
	\\
	&\ceq N^{\gamma + \tilde{\gamma}} \norm{P_{N}F}_{L^{\tilde{q}'}(I:L^{\tilde{r}'}(\R^d))},
\end{align*}
where we use $\theta \tilde{\gamma}= \gamma$. 

At second, we consider the case of $2 \leq \tilde{r} \leq r$. Then, we have $\tilde{q} \geq q$.  Let $\eta \in [0,1]$ satisfy
\begin{align*}
	\frac{1}{\tilde{q}'} = \frac{1-\eta}{1} + \frac{\eta}{q'},
	\quad
	\frac{1}{\tilde{r}'} = \frac{1-\eta}{2} + \frac{\eta}{r'}.
\end{align*}
Then, we have $\eta \gamma=\tilde{\gamma}$. By the interpolation, Lemmas \ref{lem2.6}, and \ref{lem2.8}, we get the desired inequality, where we note that $N^{(1-\eta) \gamma}N^{\eta 2\gamma} = N^{\gamma + \tilde{\gamma}}$. 
Taking summation for dyadic number $N$ gives the statement. 
\end{proof}

We can get Strichartz estimates even when $(1/q,1/r)$ and $(1/\tilde{q},1/\tilde{r})$ are not on a same line by permitting more derivative loss.

\begin{lemma}
\label{lem2.10}
Let $d \geq 2$. 
Let $2 \leq r, \tilde{r} < \infty$ and $2 \leq q, \tilde{q} \leq \infty$. We set $\gamma:= \max \{ d(1/2-1/r)-1/q, \frac{d+1}{2}(1/2-1/r) \}$ and $\tilde{\gamma}$ in the same manner. 
Assume that
\begin{align*}
	\frac{1}{\tilde{q}} \l( \frac{1}{2}- \frac{1}{r}\r) \neq \frac{1}{q} \l( \frac{1}{2} - \frac{1}{\tilde{r}} \r).
\end{align*} 
We also assume that $(q,r) \neq (2, 2(d-1)/(d-3))$ and $(\tilde{q},\tilde{r}) \neq (2, 2(d-1)/(d-3))$ when $d \geq 4$. 
Then, we have 
\begin{align*}
	\norm{\int_{0}^{t} e^{-\frac{t-s}{2}} e^{\pm i(t-s)\sqrt{-\Delta-1/4}} P_{>1}  F(s) ds}_{L^{q}(I:L^{r}(\R^d))}
	\cleq  \norm{|\nabla|^{\gamma+\tilde{\gamma}+\delta}F}_{L^{\tilde{q}'}(I:L^{\tilde{r}'}(\R^d))},
\end{align*}
where $\delta \geq 0$ is defined in the table 1 (see Proposition \ref{prop1.2}). 
Moreover, we have
\begin{align*}
	\norm{\int_{0}^{t} \cD_{h}(t-s)  F(s) ds}_{L^{q}(I:L^{r}(\R^d))}
	&\cleq  \norm{|\nabla|^{\gamma+\tilde{\gamma}+\delta} \jbra{\nabla}^{-1}F}_{L^{\tilde{q}'}(I:L^{\tilde{r}'}(\R^d))},
	\\
	\norm{\int_{0}^{t}  (\partial_t \cD_{h} ) (t-s)  F(s) ds}_{L^{q}(I:L^{r}(\R^d))}
	&\cleq  \norm{|\nabla|^{\gamma+\tilde{\gamma}+\delta}F}_{L^{\tilde{q}'}(I:L^{\tilde{r}'}(\R^d))}.
\end{align*}
\end{lemma}

\begin{proof}
We consider the following cases respectively. 
\begin{enumerate}
\item $\frac{1}{\tilde{q}}\l( \frac{1}{2}- \frac{1}{r}\r) < \frac{1}{q} \l( \frac{1}{2}- \frac{1}{\tilde{r}}\r)$
\item $\frac{1}{\tilde{q}}\l( \frac{1}{2}- \frac{1}{r}\r) > \frac{1}{q} \l( \frac{1}{2}- \frac{1}{\tilde{r}}\r)$
\end{enumerate}
\begin{enumerate}
\renewcommand{\theenumi}{\alph{enumi}}
\item $\frac{d-1}{2}\l( \frac{1}{2}- \frac{1}{r}\r) \geq \frac{1}{q}$ and $\frac{d-1}{2}\l( \frac{1}{2}- \frac{1}{\tilde{r}}\r) \geq \frac{1}{\tilde{q}}$
\item $\frac{d-1}{2}\l( \frac{1}{2}- \frac{1}{r}\r) \geq \frac{1}{q}$ and $\frac{d-1}{2}\l( \frac{1}{2}- \frac{1}{\tilde{r}}\r) < \frac{1}{\tilde{q}}$
\item $\frac{d-1}{2}\l( \frac{1}{2}- \frac{1}{r}\r) < \frac{1}{q}$ and $\frac{d-1}{2}\l( \frac{1}{2}- \frac{1}{\tilde{r}}\r) \geq \frac{1}{\tilde{q}}$
\item $\frac{d-1}{2}\l( \frac{1}{2}- \frac{1}{r}\r) < \frac{1}{q}$ and $\frac{d-1}{2}\l( \frac{1}{2}- \frac{1}{\tilde{r}}\r) < \frac{1}{\tilde{q}}$
\end{enumerate}
It is easy to show that Cases (1)-(b) and (2)-(c) do not occur. 

\noindent{\bf Case(1).} We treat the case of $\frac{1}{\tilde{q}}\l( \frac{1}{2}- \frac{1}{r}\r) < \frac{1}{q} \l( \frac{1}{2}- \frac{1}{\tilde{r}}\r)$. Since $\frac{1}{\tilde{q}}\l( \frac{1}{2}- \frac{1}{r}\r) <\frac{1}{q}\l( \frac{1}{2}- \frac{1}{\tilde{r}}\r)$, there exists $r_1 \in [2,\tilde{r})$ such that 
\begin{align*}
	\frac{1}{\tilde{q}}\l( \frac{1}{2}- \frac{1}{r}\r) =\frac{1}{q}\l( \frac{1}{2}- \frac{1}{r_1}\r).
\end{align*}
Let $\gamma_1$ be the derivative loss for the pair $(\tilde{q},r_1)$. Then, by Lemma \ref{lem2.9} and the Bernstein inequality, we get
\begin{align*}
	\norm{\Psi[F]}_{L_t^qL_x^r}
	&\cleq N^{\gamma+\gamma_1} \norm{P_N F}_{L_t^{\tilde{q}'}L_x^{r_1'}}
	\\
	&\cleq N^{\gamma+\gamma_1} N^{d\l(\frac{1}{\tilde{r}'} - \frac{1}{r_1'}\r)} \norm{P_N F}_{L_t^{\tilde{q}'}L_x^{\tilde{r}'}}.
\end{align*}

\noindent{\bf Case(1)-(a).} If $\frac{d-1}{2}\l( \frac{1}{2}- \frac{1}{r}\r) \geq \frac{1}{q}$, which also gives $\frac{d-1}{2}\l( \frac{1}{2}- \frac{1}{r_1}\r) \geq \frac{1}{\tilde{q}}$, we have $\gamma_1=d(1/2-1/{r_1})-1/{\tilde{q}}$. Thus, we obtain
\begin{align*}
	\gamma+\gamma_1+d\l(\frac{1}{\tilde{r}'} - \frac{1}{r_1'}\r)
	&=\gamma + d\l( \frac{1}{2}- \frac{1}{r_1}\r) - \frac{1}{q_1}+d\l(\frac{1}{\tilde{r}'} - \frac{1}{r_1'}\r)
	\\
	&=\gamma + d\l( \frac{1}{2}- \frac{1}{\tilde{r}}\r) - \frac{1}{\tilde{q}}
	\\
	&=\gamma + \tilde{\gamma}.
\end{align*}

\noindent{\bf Case(1)-(c).} $\frac{d-1}{2}\l( \frac{1}{2}- \frac{1}{r}\r) < \frac{1}{q}$ gives $\frac{d-1}{2}\l( \frac{1}{2}- \frac{1}{r_1}\r) < \frac{1}{\tilde{q}}$. Then, we have $\gamma_1 = \frac{d+1}{2}(1/2-1/{r_1})$. 
Moreover, since $\frac{d-1}{2}\l( \frac{1}{2}- \frac{1}{\tilde{r}}\r) \geq \frac{1}{\tilde{q}}$, we have $\tilde{\gamma}=d(1/2-1/{\tilde{r}})-1/{\tilde{q}}$. Therefore, we obtain
\begin{align*}
	\gamma+\gamma_1+d\l(\frac{1}{\tilde{r}'} - \frac{1}{r_1'}\r)
	&=\gamma + \tilde{\gamma} + \gamma_1 -\tilde{\gamma}+d\l(\frac{1}{\tilde{r}'} - \frac{1}{r_1'}\r)
	\\
	&=\gamma + \tilde{\gamma}+\frac{q}{\tilde{q}} \l\{\frac{1}{q}- \frac{d-1}{2} \l(\frac{1}{2}-\frac{1}{r}\r) \r\},
\end{align*}
where we use $q\l( \frac{1}{2}- \frac{1}{r}\r) =\tilde{q}\l( \frac{1}{2}- \frac{1}{r_1}\r)$ in the last equality. 

\noindent{\bf Case(1)-(d).} We have $\gamma_1 = \frac{d+1}{2}(1/2-1/{r_1})$ since $\frac{d-1}{2}\l( \frac{1}{2}- \frac{1}{r}\r) < \frac{1}{q}$. Since $\frac{d-1}{2}\l( \frac{1}{2}- \frac{1}{\tilde{r}}\r) < \frac{1}{\tilde{q}}$, we have $\tilde{\gamma}=\frac{d+1}{2}(1/2-1/{\tilde{r}})$ and thus we obtain
\begin{align*}
	\gamma+\gamma_1+d\l(\frac{1}{\tilde{r}'} - \frac{1}{r_1'}\r)
	&=\gamma + \tilde{\gamma} + \gamma_1 -\tilde{\gamma}+d\l(\frac{1}{\tilde{r}'} - \frac{1}{r_1'}\r)
	\\
	&=\gamma + \tilde{\gamma}+\frac{1}{\tilde{q}} \frac{d-1}{2} \l\{ \tilde{q}\l( \frac{1}{2}- \frac{1}{\tilde{r}}\r) - q\l( \frac{1}{2} - \frac{1}{r}\r) \r\},
\end{align*}
where we use $\frac{1}{\tilde{q}}\l( \frac{1}{2}- \frac{1}{r}\r) =\frac{1}{q} \l( \frac{1}{2}- \frac{1}{r_1}\r)$ in the last equality.

\noindent{\bf Case(2).} We treat the case of $\frac{1}{\tilde{q}}\l( \frac{1}{2}- \frac{1}{r}\r) > \frac{1}{q}\l( \frac{1}{2}- \frac{1}{\tilde{r}}\r)$. Since $\frac{1}{\tilde{q}}\l( \frac{1}{2}- \frac{1}{r}\r) > \frac{1}{q}\l( \frac{1}{2}- \frac{1}{\tilde{r}}\r)$, there exists $r_2 \in [2,r)$ such that
\begin{align*}
	\frac{1}{\tilde{q}}\l( \frac{1}{2}- \frac{1}{r_2}\r) =\frac{1}{q}\l( \frac{1}{2}- \frac{1}{\tilde{r}}\r).
\end{align*}
Let $\gamma_2$ be the derivative loss for the pair $(q,r_2)$. Then, by the Bernstein inequality and Lemma \ref{lem2.9}, we get
\begin{align*}
	\norm{\Psi[F]}_{L_t^qL_x^r}
	& \cleq N^{d\l(\frac{1}{r_2} - \frac{1}{r}\r)} \norm{\Psi[F]}_{L_t^{q}L_x^{r_2}}
	\\
	&\cleq N^{d\l(\frac{1}{r_2} - \frac{1}{r}\r)}  N^{\gamma_2+\tilde{\gamma}} \norm{P_N F}_{L_t^{\tilde{q}'}L_x^{\tilde{r}'}}.
\end{align*}
By the symmetric argument, we get the desired statements. 
\end{proof}

%%%%%%%%%%%%%%%%%%%%%%%%%%%%%%%%%%%%%%%%%%%%%%%%%%%%%%%

\subsection{Proof of the Strichartz estimates}
\label{sec2.3}

\begin{proof}[Proof of Proposition \ref{prop1.1}]
We only show the inequality for $\cD$ since the similar argument works for $\partial_t \cD$ and $\partial_t^2 \cD$. 
By the integral inequality, we get
\begin{align*}
	\norm{ \cD(t) f}_{L^{q}(I:L^r(\R^d))} 
	\leq \norm{ \cD_{l}(t) f}_{L^{q}(I:L^r(\R^d))}  
	+ \norm{ \cD_{h}(t) f}_{L^{q}(I:L^r(\R^d))}.
\end{align*}
By the assumption of $(q,r)$, we can apply Lemma \ref{lem2.2} to the first term as $\tilde{r}=2$ and $\sigma=1$ and Lemma \ref{lem2.4} to the second term. Then it follows that
\begin{align*}
	\norm{ \cD(t) f}_{L^{q}(I:L^r(\R^d))} 
	&\leq \norm{ \cD_{l}(t) f}_{L^{q}(I:L^r(\R^d))}  
	+ \norm{ \cD_{h}(t) f}_{L^{q}(I:L^r(\R^d))}
	\\
	&\cleq \norm{\jbra{\nabla}^{-1}f}_{L^2} + \norm{|\nabla|^{\gamma} \jbra{\nabla}^{-1}f}_{L^2}
	\\
	& \ceq \norm{f}_{H^{\gamma-1}}
\end{align*}
This finishes the proof except for the heat end-point case. Next, we show the end-point estimate. First, we prove the following lemma, which is essentially obtained by Watanabe \cite[Lemma 2.8]{Wat17}. However, we give a proof for reader's convenience. Let $d \geq 3$ and $(q,r)=(2,2d/(d-2))$ from now on in this proof. 

\begin{lemma}[Homogeneous Strichartz estimate in the heat end-point case {(see \cite[Lemma 2.8]{Wat17})}]
\label{lem2.11}
Let $d \geq 3$ and $(q,r)=(2,2d/(d-2))$. 
Then, we have
\begin{align*}
	\norm{\cD(t)f}_{L_{t}^{q} (I: L_{x}^{r}(\R^d))} 
	&\cleq \norm{f}_{L^2},
	\\
	\norm{\partial_t \cD(t)f}_{L_{t}^{q} (I: L_{x}^{r}(\R^d))} 
	&\cleq \norm{ \jbra{\nabla}f}_{L^2},
	\\
	\norm{\partial_{t}^2 \cD(t)f}_{L_{t}^{q} (I: L_{x}^{r}(\R^d)) } 
	&\cleq  \norm{\jbra{\nabla}^2 f}_{L^{2}}.
\end{align*}
\end{lemma}

\begin{proof}[Proof of Lemma \ref{lem2.11}]
We use the energy method. Let $\phi$ be the solution of the linear equation with the initial data $(\phi(0),\partial_{t} \phi(0))=(\phi_0,\phi_1)$. Multiplying $\partial_{t} \phi$ by the linear equation and integrating it on $[0,t) \times \R^{d}$, we get
\begin{align*}
	\frac{1}{2}\l( \norm{\partial_{t} \phi(t)}_{L^{2}}^{2} +\norm{\nabla \phi(t)}_{L^{2}}^{2} \r)+ \int_{0}^{t} \norm{\partial_{t} \phi(s)}_{L^{2}}^{2} ds
	\leq \norm{\phi_{0}}_{H^{1}}^{2} + \norm{\phi_{1}}_{L^{2}}^{2}.
\end{align*}
Moreover, multiplying $\phi$ by the linear equation and integrating it on $[0,t) \times \R^{d}$, we get
\begin{align*}
	- \norm{\partial_{t} \phi(t)}_{L^{2}}^{2} + \frac{1}{4} \norm{\phi(t)}_{L^{2}}^{2} 
	- \int_{0}^{t} \norm{\partial_{t} \phi(s)}_{L^{2}}^{2} ds+ \int_{0}^{t} \norm{\nabla \phi(s)}_{L^{2}}^{2} ds
	\leq \norm{\phi_{0}}_{H^{1}}^{2} + \norm{\phi_{1}}_{L^{2}}^{2}.
\end{align*}
Combining these estimates, we get
\begin{align*}
	\norm{\phi(t)}_{H^{1}}^{2} 
	+\norm{\partial_{t} \phi(t)}_{L^{2}}^{2} 
	+ \int_{0}^{t} \norm{\partial_{t} \phi(s)}_{L^{2}}^{2} ds 
	+ \int_{0}^{t} \norm{\partial_{t} \phi(s)}_{L^{2}}^{2} ds
	\cleq \norm{\phi_{0}}_{H^{1}}^{2} + \norm{\phi_{1}}_{L^{2}}^{2}.
\end{align*}
Now, by the Sobolev embedding and this inequality, we have
\begin{align*}
	\norm{\phi}_{L_{t}^{2}L_{x}^{\frac{2d}{d-2}}}^{2}
	\cleq \norm{\phi}_{L_{t}^{2}\dot{H}_{x}^{1}}^{2}
	\cleq \norm{\phi_{0}}_{H^{1}}^{2} + \norm{\phi_{1}}_{L^{2}}^{2}.
\end{align*}
Moreover, by differentiating the linear equation by $\partial_{x_i}$ for $i=1,2,\cdots,d$, multiplying $\partial_t \partial_{x_i} \phi$ and $\partial_{x_i} \phi$, and repeating the above argument, we get
\begin{align*}
	\norm{\partial_{x_i}  \phi(t)}_{H^{1}}^{2} 
	+\norm{\partial_{t} \partial_{x_i}  \phi(t)}_{L^{2}}^{2} 
	&+ \int_{0}^{t} \norm{\partial_{t} \partial_{x_i}  \phi(s)}_{L^{2}}^{2} ds 
	+ \int_{0}^{t} \norm{\partial_{t} \partial_{x_i}  \phi(s)}_{L^{2}}^{2} ds
	\\
	&\cleq \norm{\partial_{x_i}  \phi_{0}}_{H^{1}}^{2} + \norm{\partial_{x_i}  \phi_{1}}_{L^{2}}^{2}
	\\
	&\cleq \norm{\phi_{0}}_{H^{2}}^{2} + \norm{\phi_{1}}_{H^{1}}^{2}.
\end{align*} 
Thus, we can estimate $L_{t}^{2}L_{x}^{\frac{2d}{d-2}}$-norm of the time derivative of the solution as follows. 
\begin{align*}
	\norm{\partial_{t} \phi}_{L_{t}^{2}L_{x}^{\frac{2d}{d-2}}}^{2}
	\cleq \norm{\partial_{t} \phi}_{L_{t}^{2}\dot{H}_{x}^{1}}^{2}
	\cleq \norm{\phi_{0}}_{H^{2}}^{2} + \norm{\phi_{1}}_{H^{1}}^{2}.
\end{align*}
Here, $\cD(t)f$ is the solution of the linear equation with initial data $(\phi_0,\phi_1)= (0,f)$. Therefore, we obtain
\begin{align*}
	\norm{\cD(t)f}_{L_{t}^{2}L_{x}^{\frac{2d}{d-2}} }
	\cleq \norm{f}_{L^{2}}.
\end{align*}
Since $\cD(t)f+\partial_{t} \cD(t)f$ is the solution of the linear equation with initial data $(\phi_0,\phi_1)= (f,0)$, it follows from the triangle inequality that
\begin{align*}
	\norm{\partial_{t} \cD(t)f}_{L_{t}^{2}L_{x}^{\frac{2d}{d-2}} } 
	&\leq \norm{\cD(t)f+\partial_{t} \cD(t)f}_{L_{t}^{2}L_{x}^{\frac{2d}{d-2}} } +\norm{\cD(t)f}_{L_{t}^{2}L_{x}^{\frac{2d}{d-2}} } 
	\\
	&\cleq  \norm{f}_{H^{1}}+\norm{f}_{L^{2}}
	\cleq  \norm{f}_{H^{1}}.
\end{align*}
We can also estimate $\partial_t^2 \cD(t)f$ as follows since $\partial_t(\cD(t)f+\partial_{t} \cD(t)f)$ is the time derivative of the solution with the initial data $(\phi_0,\phi_1)= (f,0)$.
\begin{align*}
	\norm{\partial_{t}^2 \cD(t)f}_{L_{t}^{2}L_{x}^{\frac{2d}{d-2}} } 
	& \leq \norm{\partial_{t} (\cD(t)f+\partial_{t} \cD(t)f)}_{L_{t}^{2}L_{x}^{\frac{2d}{d-2}} } +\norm{\partial_{t} \cD(t)f}_{L_{t}^{2}L_{x}^{\frac{2d}{d-2}} } 
	\\
	&\cleq  \norm{f}_{H^{2}}+\norm{f}_{H^{1}}
	\cleq  \norm{f}_{H^{2}}.
\end{align*}
This completes the proof of Lemma \ref{lem2.11}. 
\end{proof}
By the first estimate in Lemma \ref{lem2.11}, we have
\begin{align*}
	\norm{\jbra{\nabla}^{\sigma} \cD(t) P_{\leq1}f}_{L_{t}^{q} (I: L_{x}^{r}(\R^d))} 
	\cleq \norm{\jbra{\nabla}^{\sigma} P_{\leq1} f}_{L^2} \cleq \norm{f}_{L^2},
\end{align*}
for $\sigma\geq 0$. Therefore, it follows from this inequality and Lemma \ref{lem2.4} that 
\begin{align*}
	\norm{\cD(t) f}_{L_{t}^{q} (I: L_{x}^{r}(\R^d))}
	&\leq \norm{\cD_{l}(t) f}_{L_{t}^{q} (I: L_{x}^{r}(\R^d))} + \norm{\cD_{h}(t) f}_{L_{t}^{q} (I: L_{x}^{r}(\R^d))}
	\\
	&\cleq \norm{\jbra{\nabla}^{-1} f}_{L^{2}}   + \norm{|\nabla|^{\gamma} \jbra{\nabla}^{-1}f}_{L^2}
	\\
	& \ceq \norm{f}_{H^{\gamma-1}}.
\end{align*}
This completes the proof of the heat end-point homogeneous Strichartz estimate.
\end{proof}

\begin{proof}[Proof of Proposition \ref{prop1.2}]
We only show the inequality for $\cD$ since the similar argument works for $\partial_t \cD$. 
By the integral inequality, we get
\begin{align*}
	&\norm{\int_{0}^{t} \cD(t-s)  F(s) ds}_{L^{q}(I:L^{r}(\R^d))}
	\\
	&\leq \norm{\int_{0}^{t} \cD_{l}(t-s)  F(s) ds}_{L^{q}(I:L^{r}(\R^d))} 
	+\norm{\int_{0}^{t} \cD_{h}(t-s)  F(s) ds}_{L^{q}(I:L^{r}(\R^d))}
\end{align*}
By the assumption of $(q,r)$, we can apply Lemma \ref{lem2.3} to the first term as $\tilde{r}=2$ and $\sigma=1$ and Lemmas \ref{lem2.9}, \ref{lem2.10} to the second term. Then it follows that
\begin{align*}
	&\norm{\int_{0}^{t} \cD(t-s)  F(s) ds}_{L^{q}(I:L^{r}(\R^d))}
	\\
	&\leq \norm{\int_{0}^{t} \cD_{l}(t-s)  F(s) ds}_{L^{q}(I:L^{r}(\R^d))} 
	+\norm{\int_{0}^{t} \cD_{h}(t-s)  F(s) ds}_{L^{q}(I:L^{r}(\R^d))}
	\\
	&\cleq  \norm{\jbra{\nabla}^{-1} F}_{L^{\tilde{q}'}(I:L^{\tilde{r}'}(\R^d))}
	+ \norm{|\nabla|^{\gamma+\tilde{\gamma}+\delta} \jbra{\nabla}^{-1}F}_{L^{\tilde{q}'}(I:L^{\tilde{r}'}(\R^d))}
	\\
	&\ceq  \norm{F}_{L^{\tilde{q}'}(I:W^{\gamma+\tilde{\gamma}+\delta-1,\tilde{r}'}(\R^d))}
\end{align*}
This is the desired estimate. 
\end{proof}

%%%%%%%%%%%%%%%%%%%%%%%%%%%%%%%%%%%%%%%%%%%%%%%%%%%%%%%%
%%%%%%%%%%%%%%%%%%%%%%%%%%%%%%%%%%%%%%%%%%%%%%%%%%%%%%%%

\section{Well-posedness for the energy critical nonlinear damped wave equation}
\label{sec3}

In this section, we prove local well-posedness for \eqref{NLDW}, Theorem \ref{thm1.3}, by contraction mapping principle. 
We define the complete metric space
\begin{align*}
	X(T,L,M):=\l\{ v \text{ on } [0,T) \times \R^d  : \norm{\jbra{\nabla}^{\frac{1}{2}} v }_{L_{t,x}^{\frac{2(d+1)}{d-1}}([0,T))} \leq L, \norm{v}_{L_{t,x}^{\frac{2(d+1)}{d-2}}([0,T))} \leq M \r\}.
\end{align*}

\begin{remark}
\label{rem3.1}
$(q,r)=(2(d+1)/(d-1),2(d+1)/(d-1))$ and $(2(d+1)/(d-2),2(d+1)/(d-2))$ satisfy the assumptions of the Strichartz estimates in Propositions \ref{prop1.1} and \ref{prop1.2}. Moreover, $\gamma=1/2$ when $(q,r)=(2(d+1)/(d-1),2(d+1)/(d-1))$ and $\gamma=1$ when $(q,r)=(2(d+1)/(d-2),2(d+1)/(d-2))$. We note that these exponents are same as in the local well-posedness for the critical nonlinear wave equation. 
\end{remark}

We define 
\begin{align*}
	\Phi[u](t)=\Phi_{u_0,u_1}[u](t):= \cD(t) (u_0+u_1) +\partial_t \cD(t) u_0 + \int_{0}^{t} \cD(t-s) \cN(u(s)) ds.
\end{align*}

\begin{proof}[Proof of Theorem \ref{thm1.3}]
As stated in Remark \ref{rem3.1}, the exponents are same as in the argument for the energy critical nonlinear wave equation. Thus, the proof if similar so that we only give sketch of the proof. See \cite{Pec84,GSV92,ShSt93,KeMe08} for details. 
Since $(u_0,u_1) \in H^1(\R^d) \times L^2(\R^d)$, by the Strichartz estimates in Proposition \ref{prop1.1}, we obtain 
\begin{align}
	\label{eq3.1}
	&\norm{ \cD(t) (u_0+u_1) +\partial_t \cD(t) u_0}_{X(T)}
	\\ \notag
	&\leq \norm{\jbra{\nabla}^{\frac{1}{2}} \cD(t) (u_0+u_1) }_{L_{t,x}^{\frac{2(d+1)}{d-1}}([0,T))}
	+ \norm{\jbra{\nabla}^{\frac{1}{2}} \partial_t \cD(t) u_0}_{L_{t,x}^{\frac{2(d+1)}{d-1}}}
	\\ \notag
	&\quad +\norm{ \cD(t) (u_0+u_1) }_{L_{t,x}^{\frac{2(d+1)}{d-2}}([0,T))}
	+\norm{ \partial_t \cD(t) u_0}_{L_{t,x}^{\frac{2(d+1)}{d-2}}([0,T))}
	\\ \notag
	&\cleq \norm{u_0}_{H^1} + \norm{u_1}_{L^2} <A< \infty.
\end{align}
%Therefore, for sufficiently small $T>0$, we can get
%\begin{align}
%	\label{eq3.2}
%	\norm{ \cD(t) (u_0+u_1) +\partial_t \cD(t) u_0}_{X(T)} 
%	\leq \frac{1}{2} \min\{L,M\}. 
%\end{align}
We estimate the nonlinear term as follows. By the Strichartz estimates in Proposition \ref{prop1.2} and the fractional Leibnitz rule (see \cite[Lemma 2.5]{KeMe08} and references therein), we get 
\begin{align}
\label{eq3.3}
	&\norm{\jbra{\nabla}^{\frac{1}{2}} \int_{0}^{t} \cD(t-s) \cN(u(s)) ds}_{L_{t,x}^{\frac{2(d+1)}{d-1}}([0,T))}
	\\ \notag
	&\quad \cleq \norm{\jbra{\nabla}^{\frac{1}{2}} \cN(u)}_{L_{t,x}^{\frac{2(d+1)}{d+3}}([0,T))}
	\\ \notag
	&\quad \cleq \norm{u}_{L_{t,x}^{\frac{2(d+1)}{d-2}}([0,T))}^{\frac{4}{d-2}}
	\norm{\jbra{\nabla}^{\frac{1}{2}}  u}_{L_{t,x}^{\frac{2(d+1)}{d-1}}([0,T))}
\end{align}
and 
\begin{align}
\label{eq3.4}
	&\norm{ \int_{0}^{t} \cD(t-s) \cN(u(s)) ds}_{L_{t,x}^{\frac{2(d+1)}{d-2}}([0,T))}
	\\ \notag
	& \quad \cleq \norm{\jbra{\nabla}^{\frac{1}{2}} \cN(u)}_{L_{t,x}^{\frac{2(d+1)}{d+3}}([0,T))}
	\\ \notag
	&\quad \cleq \norm{u}_{L_{t,x}^{\frac{2(d+1)}{d-2}}([0,T))}^{\frac{4}{d-2}}
	\norm{\jbra{\nabla}^{\frac{1}{2}}  u}_{L_{t,x}^{\frac{2(d+1)}{d-1}}([0,T))}.
\end{align}
Combining \eqref{eq3.1} and \eqref{eq3.3}, we obtain
\begin{align*}
	\norm{\jbra{\nabla}^{\frac{1}{2}} \Phi[u]}_{L_{t,x}^{\frac{2(d+1)}{d-1}}([0,T))}
	&\leq \norm{\jbra{\nabla}^{\frac{1}{2}} \cD(t) (u_0+u_1) +\partial_t \cD(t) u_0}_{L_{t,x}^{\frac{2(d+1)}{d-1}}([0,T))} 
	\\
	&\quad+ \norm{\jbra{\nabla}^{\frac{1}{2}} \int_{0}^{t} \cD(t-s) \cN(u(s)) ds}_{L_{t,x}^{\frac{2(d+1)}{d-1}}([0,T))}
	\\
	&\leq CA+ C L M^{\frac{4}{d-2}}
	\\
	&\leq L
\end{align*}
%if we choose $M$ such that $M \leq (2C_0)^{-(d-2)/4}$. 
if we choose $L=2CA$ and $M$ such that $CM^{4/(d-2)} \leq 1/2$. 
By \eqref{eq3.1} and \eqref{eq3.4}, we get
\begin{align*}
	\norm{\Phi[u]}_{L_{t,x}^{\frac{2(d+1)}{d-2}}([0,T))}
	&\leq \norm{\cD(t) (u_0+u_1) +\partial_t \cD(t) u_0}_{L_{t,x}^{\frac{2(d+1)}{d-2}}([0,T))} 
	\\
	&\quad+ \norm{ \int_{0}^{t} \cD(t-s) \cN(u(s)) ds}_{L_{t,x}^{\frac{2(d+1)}{d-2}}([0,T))}
	\\
	&\leq \delta + C L M^{\frac{4}{d-2}}
	\\
	&\leq M
\end{align*}
if we choose $\delta = M/2$ and $ L \leq(2C)^{-1}M^{(d-6)/(d-2)}$ (which is possible if $d \leq5$). Thus, $\Phi$ is a mapping on $X(T,L,M)$. 
\begin{align*}
	&\norm{\jbra{\nabla}^{\frac{1}{2}} (\Phi[u]-\Phi[v])}_{L_{t,x}^{\frac{2(d+1)}{d-1}}([0,T))}
	+\norm{\Phi[u]-\Phi[v]}_{L_{t,x}^{\frac{2(d+1)}{d-2}}([0,T))}
	\\
	& \cleq \norm{\jbra{\nabla}^{\frac{1}{2}} (\cN(u)-\cN(v))}_{L_{t,x}^{\frac{2(d+1)}{d+3}}([0,T))}
	\\
	& \cleq \l( \norm{u}_{L_{t,x}^{\frac{2(d+1)}{d-2}}([0,T))}^{\frac{4}{d-2}} + \norm{v}_{L_{t,x}^{\frac{2(d+1)}{d-2}}([0,T))}^{\frac{4}{d-2}} \r)
	\norm{\jbra{\nabla}^{\frac{1}{2}} ( u-v)}_{L_{t,x}^{\frac{2(d+1)}{d-1}}([0,T))}
	\\
	&\quad +\l( \norm{u}_{L_{t,x}^{\frac{2(d+1)}{d-2}}([0,T))}^{\frac{6-d}{d-2}} + \norm{v}_{L_{t,x}^{\frac{2(d+1)}{d-2}}([0,T))}^{\frac{6-d}{d-2}} \r)
	\\
	&\quad \quad \times \l( \norm{\jbra{\nabla}^{\frac{1}{2}} u}_{L_{t,x}^{\frac{2(d+1)}{d-1}}([0,T))} +\norm{\jbra{\nabla}^{\frac{1}{2}} v}_{L_{t,x}^{\frac{2(d+1)}{d-1}}([0,T))} \r)
	\\
	&\quad \quad \times  \norm{u-v}_{L_{t,x}^{\frac{2(d+1)}{d-2}}([0,T))}
	\\
	& \leq C M^{\frac{4}{d-2}}  \norm{\jbra{\nabla}^{\frac{1}{2}} ( u-v)}_{L_{t,x}^{\frac{2(d+1)}{d-1}}([0,T))} 
	+ CM^{\frac{6-d}{d-2}} L \norm{u-v}_{L_{t,x}^{\frac{2(d+1)}{d-2}}([0,T))}. 
\end{align*}
Taking $L$ and $M$ sufficiently small, $\Phi$ is a contraction mapping on $X(T,L,M)$. By the Banach fixed point theorem, we obtain the solution such that $u=\Phi[u]$. 
Then, $(u,\partial_t u)$ belongs to $C([0,T); H^1(\R^d) \times L^2(\R^d))$ because of the Strichartz estimates (Proposition \ref{prop1.1} and \ref{prop1.2}) and the nonlinear estimates (for example $\jbra{\nabla}^{\frac{1}{2}} \cN(u) \in L_{t,x}^{\frac{2(d+1)}{d+3}}$). 
We give a proof of the standard blow-up criterion. We suppose that $T_{+}=T_{+}(u_0,u_1)<\infty$ and $\norm{u}_{L_{t,x}^{2(d+1)/d-2}([0,T_{+}))}<\infty$. Take $\tau$ and $T$ arbitrary such that $0<\tau<T<T_{+}$. By the Duhamel formula, we have
\begin{align*}
	u(t) =  \cD(t-\tau) (u(\tau)+\partial_t u(\tau)) +\partial_t \cD(t-\tau) u(\tau) + \int_{\tau}^{t} \cD(t-s) \cN(u(s)) ds,
\end{align*}
for $t>\tau$. By the Strichartz estimates, we obtain
\begin{align*}
	& \norm{\jbra{\nabla}^{\frac{1}{2}} u}_{L_{t,x}^{\frac{2(d+1)}{d-1}}((\tau,T))}
	\\
	&\cleq \norm{ (u(\tau),\partial_t u(\tau))}_{H^1 \times L^2}
	+ \norm{ \int_{\tau}^{t} \jbra{\nabla}^{\frac{1}{2}} \cD(t-s) \cN(u(s)) ds}_{L_{t,x}^{\frac{2(d+1)}{d-1}}((\tau,T))}
	\\
	&\cleq \norm{(u(\tau),\partial_t u(\tau))}_{H^1\times L^2}+ \norm{ \jbra{\nabla}^{\frac{1}{2}} \cN(u(s)) }_{L_{t,x}^{\frac{2(d+1)}{d+3}}((\tau,T))}
	\\
	&\cleq \norm{(u(\tau),\partial_t u(\tau))}_{H^1\times L^2}
	+ \norm{u}_{L_{t,x}^{\frac{2(d+1)}{d-2}}((\tau,T))}^{\frac{4}{d-2}} \norm{ \jbra{\nabla}^{\frac{1}{2}} u }_{L_{t,x}^{\frac{2(d+1)}{d-1}}((\tau,T))}.
\end{align*}
Since $\norm{u}_{L_{t,x}^{\frac{2(d+1)}{d-2}}((\tau,T))} \ll 1$ for $\tau$ close to $T_{+}$, we obtain 
\begin{align*}
	\norm{ \jbra{\nabla}^{\frac{1}{2}} u }_{L_{t,x}^{\frac{2(d+1)}{d-1}}((\tau,T))}
	\cleq \norm{(u(\tau),\partial_t u(\tau))}_{H^1\times L^2}.
\end{align*} 
Fix such $\tau$. Since $T$ is arbitrary, we get 
\begin{align}
\label{eq3.5}
	\norm{ \jbra{\nabla}^{\frac{1}{2}} u }_{L_{t,x}^{\frac{2(d+1)}{d-1}}((\tau,T_{+}))}
	\cleq \norm{(u(\tau),\partial_t u(\tau))}_{H^1\times L^2}.
\end{align} 
Take a sequence $\{t_n\}$ such that $t_n \to T_{+}$ and $t_n >\tau$. Then, by the integral formula, the Strichartz estimates the assumption, and \ref{eq3.5}, we have
\begin{align*}
	&\norm{ \cD(t-t_n) (u(t_n)+\partial_t u(t_n)) +\partial_t \cD(t-t_n) u(t_n)}_{L_{t,x}^{\frac{2(d+1)}{d-2}}([t_n,T_{+}))}
	\\
	&\cleq \norm{u}_{L_{t,x}^{\frac{2(d+1)}{d-2}}([t_n,T_{+}))}
	+ \norm{ \int_{t_n}^{t} \cD(t-s) \cN(u(s)) ds}_{L_{t,x}^{\frac{2(d+1)}{d-2}}([t_n,T_{+}))}
	\\
	&\cleq \norm{u}_{L_{t,x}^{\frac{2(d+1)}{d-2}}([t_n,T_{+}))}
	+ \norm{u}_{L_{t,x}^{\frac{2(d+1)}{d-2}}([t_n,T_{+}))}^{\frac{4}{d-2}}
	\norm{\jbra{\nabla}^{\frac{1}{2}}  u}_{L_{t,x}^{\frac{2(d+1)}{d-1}}([t_n,T_{+}))}
	\\
	& \to 0 \text{ as } n \to \infty,
\end{align*}
Thus, $\norm{ \cD(t-t_n) (u(t_n)+\partial_t u(t_n)) +\partial_t \cD(t-t_n) u(t_n)}_{L_{t,x}^{2(d+1)/(d-2)}([t_n,T_{+}))}<\delta/2$ is true for large $n$. Then, for some $\eps>0$, we get 
\begin{align*}
	\norm{ \cD(t-t_n) (u(t_n)+\partial_t u(t_n)) +\partial_t \cD(t-t_n) u(t_n)}_{L_{t,x}^{2(d+1)/(d-2)}([t_n,T_{+}+\eps))}<\delta.
\end{align*}
The local well-posedness derives a contradiction. 
\end{proof}

%%%%%%%%%%%%%%%%%%%%%%%%%%%%%%%%%%%%%%%%%%%%%%%%%%%%%%%

\section{Decay of global solution with finite Strichartz norm}
\label{sec4}

In this section, we give a proof of Theorem \ref{thm1.4}. 

\begin{lemma}
If $u$ is a global solution of \eqref{NLDW} with $\| u \|_{L_{t,x}^{2(d+1)/(d-2)}([0,\infty))}<\infty$, then $u$ satisfies
\begin{align*}
	\norm{\jbra{\nabla}^{\frac{1}{2}} u}_{L_{t,x}^{\frac{2(d+1)}{d-1}}([0,\infty))}<\infty
\end{align*}
\end{lemma}

\begin{proof}
The proof is very similar to the proof of the standard blow-up criterion. 
Take $0 < \tau < T <\infty$ arbitrary. We know that the global solution belongs to $L_{t,x}^{\frac{2(d+1)}{d-1}}(K)$ for any compact interval $K \subset [0,\infty)$. 
It follows from  the Duhamel's formula and the Strichartz estimates that
\begin{align*}
	\norm{\jbra{\nabla}^{\frac{1}{2}} u}_{L_{t,x}^{\frac{2(d+1)}{d-1}}((\tau,T))}
	&\cleq \norm{(u(\tau),\partial_t u(\tau))}_{H^1\times L^2}
	\\
	&\quad + \norm{u}_{L_{t,x}^{\frac{2(d+1)}{d-2}}((\tau,T))}^{\frac{4}{d-2}} \norm{ \jbra{\nabla}^{\frac{1}{2}} u }_{L_{t,x}^{\frac{2(d+1)}{d-1}}((\tau,T))}.
\end{align*}
Since $\norm{u}_{L_{t,x}^{\frac{2(d+1)}{d-2}}((\tau,T))} \ll 1$ for large $\tau$, we obtain 
\begin{align*}
	\norm{ \jbra{\nabla}^{\frac{1}{2}} u }_{L_{t,x}^{\frac{2(d+1)}{d-1}}((\tau,T))}
	\cleq \norm{(u(\tau),\partial_t u(\tau))}_{H^1\times L^2}
\end{align*} 
for large $\tau>0$. Fix such $\tau$. Since $T$ is arbitrary, we get $\norm{ \jbra{\nabla}^{\frac{1}{2}} u }_{L_{t,x}^{2(d+1)/(d-1)}((\tau,\infty))}\cleq \norm{(u(\tau),\partial_t u(\tau))}_{H^1\times L^2}$. Thus, we obtain $\norm{ \jbra{\nabla}^{\frac{1}{2}} u }_{L_{t,x}^{2(d+1)/(d-1)}([0,\infty))}<\infty$.
\end{proof}

\begin{proof}[Proof of Theorem \ref{thm1.4}]
We have
\begin{align*}
	\begin{pmatrix}
	u 
	\\ 
	\partial_t u
	\end{pmatrix}
	=\cA(t) 
	\begin{pmatrix}
	u_0
	\\
	u_1
	\end{pmatrix}
	+\int_{0}^{t}
	\cA(t-s) 
	\begin{pmatrix}
	0
	\\
	\cN(u(s))
	\end{pmatrix}
	ds,
\end{align*}
where 
\begin{align*}
	\cA(t)=
	\begin{pmatrix}
	\cD(t)+\partial_t \cD(t) & \cD(t)
	\\
	\partial_t \cD(t)+\partial_t^2 \cD(t) &\partial_t  \cD(t)
	\end{pmatrix}.
\end{align*}
We set 
\begin{align*}
	I
	&:=\cA(t) 
	\begin{pmatrix}
	u_0
	\\
	u_1
	\end{pmatrix},
	\\
	I\!\!I
	 &:= \int_{0}^{\tau}
	\cA(t-s) 
	\begin{pmatrix}
	0
	\\
	\cN(u(s))
	\end{pmatrix}
	ds,
	 \\
	 I\!\!I\!\!I
	 &:= \int_{\tau}^{t}
	\cA(t-s) 
	\begin{pmatrix}
	0
	\\
	\cN(u(s))
	\end{pmatrix}
	ds.
\end{align*}
We begin with the estimate of $I$. Approximating $(u_0,u_1)$ by $(\psi_0,\psi_1) \in (C_{0}^{\infty}(\R^d))^2$ in $H^1(\R^d) \times L^2(\R^d)$, we obtain
\begin{align*}
	\norm{I}_{H^1 \times L^2} 
	&= \norm{\cA (t) (u_0,u_1)^{T}}_{H^1 \times L^2}
	\\
	&\leq \norm{\cA (t) \{(u_0,u_1)- (\psi_0,\psi_1)\}^{T}}_{H^1 \times L^2} 
	+ \norm{\cA (t) (\psi_0,\psi_1)^{T}}_{H^1 \times L^2},
\end{align*} 
where $^{T}$ denotes transposition. 
By  \cite[Theorem 1.1]{IIOWp}, we have the following $L^p$-$L^q$ type estimates:
\begin{align*}
	\norm{\cD(t)f}_{H^1} 
	&\cleq \jbra{t}^{-\frac{d}{2} \l( \frac{1}{q}-\frac{1}{2}\r)} \norm{f}_{L^q}+ e^{-\frac{t}{2}} \jbra{t}^{\delta} \norm{f}_{L^2},
	\\
	\norm{\partial_t \cD(t)f}_{L^2} 
	&\cleq \jbra{t}^{-\frac{d}{2} \l( \frac{1}{q}-\frac{1}{2}\r)-1} \norm{f}_{L^q}+ e^{-\frac{t}{2}} \jbra{t}^{\delta} \norm{f}_{L^2},
	\\
	\norm{\partial_t \cD(t)f}_{H^1} 
	&\cleq \jbra{t}^{-\frac{d}{2} \l( \frac{1}{q}-\frac{1}{2}\r)-1} \norm{f}_{W^{1,q}}+ e^{-\frac{t}{2}} \jbra{t}^{\delta} \norm{f}_{H^1},
	\\
	\norm{\partial_t^2 \cD(t)f}_{L^2} 
	&\cleq \jbra{t}^{-\frac{d}{2} \l( \frac{1}{q}-\frac{1}{2}\r)-2} \norm{f}_{L^q}+ e^{-\frac{t}{2}} \jbra{t}^{\delta} \norm{\jbra{\nabla}f}_{L^2},
\end{align*}
for any $q \in [1,2]$ and some $\delta>0$. Therefore, applying these as $q=2$, we get
\begin{align*}
	 \norm{\cA (t) \{(u_0,u_1)- (\psi_0,\psi_1)\}^{T}}_{H^1 \times L^2}  \cleq \norm{(u_0,u_1)- (\psi_0,\psi_1)}_{H^1\times L^2}.
\end{align*}
Thus, this can be made arbitrary small by the approximation. Applying the above $L^p$-$L^q$ type estimates as $q=1$, we obtain
\begin{align*}
	\norm{\cA (t) (\psi_0,\psi_1)^{T}}_{H^1 \times L^2}
	&\leq  \norm{\cD(t) (\psi_0 +\psi_1)}_{H^1} +  \norm{\partial_t \cD(t) \psi_0 }_{H^1}
	\\
	&\quad + \norm{\partial_t \cD(t) (\psi_0 +\psi_1)}_{L^2} +  \norm{\partial_t^2 \cD(t) \psi_0 }_{L^2}
	\\
	&\cleq \jbra{t}^{-\frac{d}{4}} ( \norm{\psi_0}_{W^{1,1}}+\norm{\psi_1}_{L^1} ) 
	+ e^{-\frac{t}{4}} (\norm{\psi_0}_{H^1}+ \norm{\psi_1}_{L^2} )
	 \\
	 &\to 0 \text{ as } t \to \infty. 
\end{align*}
Next, we consider the estimate of $I\!\!I\!\!I$. By the Strichartz estimates, we have
\begin{align}
\label{eq5.1}
	\norm{I\!\!I\!\!I}_{H^1 \times L^2}
	&= \norm{\jbra{\nabla}  \int_{\tau}^{t} \cD(t-s) \cN(u(s)) ds}_{L^2} 
	\\ \notag
	&\quad +\norm{ \int_{\tau}^{t} \partial_t \cD(t-s) \cN(u(s)) ds}_{L^2} 
	\\ \notag
	&\cleq  \norm{\jbra{\nabla}^{\frac{1}{2}} \cN(u)}_{L_{t,x}^{\frac{2(d+1)}{d+3}}((\tau,t))}
	\\ \notag
	&\cleq \norm{u}_{L_{t,x}^{\frac{2(d+1)}{d-2}}((\tau,t))}^{\frac{4}{d-2}}
	\norm{\jbra{\nabla}^{\frac{1}{2}}  u}_{L_{t,x}^{\frac{2(d+1)}{d-1}}((\tau,t))}.
\end{align}
Therefore, the term is arbitrary small taking $\tau$ sufficiently close to $t$. At Last, we calculate $I\!\!I$. We note that
\begin{align*}
	I\!\!I
	=   \int_{0}^{\tau}
	\cA(t-s) 
	\begin{pmatrix}
	0
	\\
	\cN(u(s))
	\end{pmatrix}
	ds
	=
	\cA(t-\tau) 
	 \int_{0}^{\tau}
	\cA(\tau-s) 
	\begin{pmatrix}
	0
	\\
	\cN(u(s))
	\end{pmatrix}
	ds.
\end{align*}
Since by \eqref{eq5.1} we know 
\begin{align*}
	\int_{0}^{\tau} \cA(\tau-s) \begin{pmatrix} 0 \\ \cN(u(s)) \end{pmatrix}\in H^1(\R^d) \times L^2(\R^d),
\end{align*} approximating it by $\vec{\psi} \in (C_{0}^{\infty}(\R^d)^2)$, we obtain
\begin{align*}
	\norm{I\!\!I}_{H^1 \times L^2} 
	&\leq   
	\norm{
	\cA(t-\tau) 
	\l\{
	 \int_{0}^{\tau}
	\cA(\tau-s) 
	\begin{pmatrix}
	0
	\\
	\cN(u(s))
	\end{pmatrix}
	ds
	- \vec{\psi}
	\r\}
	}_{H^1 \times L^2}
	\\
	& \quad + 
	\norm{\cA(t-\tau) \vec{\psi}}_{H^1 \times L^2}.
\end{align*}
In the smae way as $I$, the first term is arbitrary small by the approximation and the second term tends to $0$ as $t \to \infty$. Combining the estimates of $I$, $I\!\!I$, and $I\!\!I\!\!I$, we get the decay. 
\end{proof}

\section{Blow-up result}
\label{sec5}

This section is devoted to the proof of the blow-up result.
The proof is essentially given by Ohta \cite{Oht97}. However, we give the full proof for the reader's convenience. In \cite{Oht97}, Ohta used argument of ordinary differential inequality instead of concavity argument. Indeed, the following lemma was used in \cite{Oht97}. See Li--Zhou \cite{LiZh95} and Souplet \cite{Sou95} for the proof.

\begin{lemma}[Li--Zhou \cite{LiZh95}, Souplet \cite{Sou95}]
\label{lem5.1}
Let $h$ satisfy 
\begin{align*}
\l\{
\begin{array}{ll}
	h''(t) + h(t) \geq C h^\gamma(t), &  t>0,
	\\
	h(0)>0, \quad h'(0)>0, 
\end{array}
\r.
\end{align*}
for some constant $C>0$ and $\gamma>1$. Then $h$ can not exist for all $t>0$. 
\end{lemma}

Setting 
\begin{align*}
	I(t):=\frac{1}{2} \norm{u(t)}_{L^2}^2, 
\end{align*}
we will prove that $I$ satisfies the ordinary differential inequality for large $t$. 

First, we show the positivity of $K$ near $0$. 

\begin{lemma}[Positivity of $K$ near $0$]
Let $\{u_n\}_{n \in \N} \subset H^1(\R^d)$ satisfy $u_n \to 0$ strongly in $H^1(\R^d)$. Then, for large $n \in \N$, we have
\begin{align*}
	K(u_n) \geq 0. 
\end{align*}
\end{lemma}

\begin{proof}
By the Sobolev inequality, we get
\begin{align*}
	K(u_n) &= \norm{\nabla u_n}_{L^2}^2 - \norm{u_n}_{L^{\frac{2d}{d-2}}}^{\frac{2d}{d-2}}
	\\
	&\geq \norm{\nabla u_n}_{L^2}^2 \l(1- C \norm{\nabla u_n}_{L^2}^{\frac{4}{d-2}}\r). 
\end{align*}
Since $u_n \to 0$ strongly in $H^1(\R^d)$, we have $K(u_n) \geq \frac{1}{2} \norm{\nabla u_n}_{L^2}^2 \geq 0$ for large $n$. 
\end{proof}

Secondly, we prove that the set $\scB$ is invariant under the flow. 

\begin{lemma}
Let $(u_0,u_1)$ belong to $\scB$. Then the solution $(u(t),\partial_t u(t))$ of \eqref{NLDW} belongs to $\scB$ for all existence time $t \in [0,T_{\max})$. 
\end{lemma}

\begin{proof}
Since the energy $E$ satisfies that 
\begin{align}
	\label{eq6.1}
	\frac{d}{dt} E(u(t),\partial_t u(t)) =-  \norm{\partial_t u(t)}_{L^2}^2, \text{ for all } t \in (0,T_{\max}),
\end{align}
we have $E(u(t), \partial_t u(t)) \leq E(u_0,u_1) \leq \mu$ for all $t \in [0,T_{\max})$. Thus, it is enough to prove $K(u(t)) < 0$ for all $t \in [0,T_{\max})$. We suppose that there exists a time $t_0 \in (0,T_{\max})$ such that $K(u(t_0))>0$. Then, by the continuity of the flow, there exists $t_* \in (0,T_{\max})$ such that $K(u(t_*))=0$ and $K(u(t))<0$ for $t \in (0,t_*)$. Assume that $u(t_*) \neq 0$. Then, by the definition of the minimizing problem $\mu$, we have
\begin{align*}
	\mu \leq J(u(t_*)) \leq E(u(t_*),\partial_t u(t_*)) \leq E(u_0,u_1) < \mu.
\end{align*}
This is a contradiction. Therefore, we get $u(t_*)=0$. Since the flow is continuous, if we take $\{t_n\} \subset (0,t_*)$ such that $t_n \to t_*$, then $u(t_n) \to u(t_*)=0$ strongly in $H^1(\R^d)$. By the positivity of $K$ near $0$, we have $K(u(t_n))\geq 0$ and $t_n \in (0,t_*)$ for large $n \in \N$. This contradicts $K(u(t))<0$ for all $t \in (0,t_*)$. Therefore, we get $K(u(t)) < 0$ for all $t \in [0,T_{\max})$. 
\end{proof}

\begin{remark}
We define
\begin{align*}
	\scG:= \{ (u_0,u_1) \in H^1(\R^d) \times L^2(\R^d): E(u_0,u_1) < \mu, K(u_0)\geq 0\}.
\end{align*}
If $(u_0,u_1)$ belongs to $\scG$, then the solution $(u(t),\partial_t u(t))$ belongs to $\scG$ for all existence time $t \in [0,T_{\max})$. This can be proved in the similar argument to the above. 
\end{remark}

We set 
\begin{align*}
	H(\varphi):= J(\varphi)  -  \frac{d-2}{2d} K(\varphi)
	=\frac{1}{d} \norm{\nabla \varphi}_{L^2}^2. 
\end{align*}

\begin{lemma}
We have
\begin{align*}
	\mu= \inf \l\{H(\varphi) : \varphi \in H^1(\R^d) \setminus \{0\}, K(\varphi) \leq 0\r\}. 
\end{align*}
\end{lemma}

\begin{proof}
We denote the right hand side by $\mu'$. It is trivial that $\mu \geq \mu'$. We prove $\mu \leq \mu'$. 
If $\varphi \in H^1(\R^d)\setminus\{0\}$ satisfies $K(\varphi)=0$, it follows that
\begin{align*}
	\mu \leq J(\varphi) = H(\varphi).
\end{align*}
If $\varphi \in H^1(\R^d)\setminus\{0\}$ satisfies $K(\varphi)<0$, there exists $\lambda_0 \in (0,1)$ such that $K(\lambda_0 \varphi)=0$. Thus, we have
\begin{align*}
	\mu \leq J(\lambda_0 \varphi) = H(\lambda_0 \varphi) < H(\varphi).
\end{align*}
Therefore, we obtain $\mu \leq H(\varphi)$ for any $\varphi \in H^1(\R^d)\setminus\{0\}$. Take the infimum, we get $\mu \leq \mu'$. 
\end{proof}

\begin{lemma}
\label{lem5.5}
Let $u(t)$ be a solution to \eqref{NLDW} on $[0,T)$ with $(u_0,u_1) \in \scB$. Then,  we have 
\begin{align*}
	I''(t)+ I'(t) \geq \l(1+\frac{d}{d-2}\r) \norm{\partial_t u(t)}_{L^2}^2 + \frac{2d}{d-2} \l(\mu - E(u(t),\partial_t u(t))\r)
\end{align*}
for all $t \in (0,T)$. 
\end{lemma}

\begin{proof}
By direct calculations and the equation, we have
\begin{align*}
	I'(t)&=\re \tbra{u(t)}{\partial_t u(t)}_{L^2},
	\\
	I''(t)&=\norm{\partial_t u(t)}_{L^2}^2 +\re \tbra{u(t)}{\partial_t^2 u(t)}_{L^2}
	\\
	&=\norm{\partial_t u(t)}_{L^2}^2 - \norm{\nabla u(t)}_{L^2}^2 - \re \tbra{u(t)}{\partial_t u(t)}_{L^2} + \norm{u(t)}_{L^{\frac{2d}{d-2}}}^{\frac{2d}{d-2}},
\end{align*}
where $\tbra{u}{v}:= \int_{\R^d} u(x) \overline{v(x)} dx$
Therefore, we obtain
\begin{align*}
	I''(t) + I'(t) = \norm{\partial_t u(t)}_{L^2}^2 - K(u(t)). 
\end{align*}
Since $K(\varphi)=\frac{2d}{d-2} (J(\varphi) - H(\varphi))=\frac{2d}{d-2} (E(\varphi,\psi) - \frac{1}{2}\norm{\psi}_{L^2}^2 - H(\varphi)) $, it follows that
\begin{align*}
	I''(t) + I'(t) 
	= \l( 1+ \frac{d}{d-2} \r) \norm{\partial_t u(t)}_{L^2}^2 + \frac{2d}{d-2} \l( H(u(t))- E(u(t), \partial_t u(t)) \r)
\end{align*}
Since we have $u(t) \in \scB$ for all $t \in [0,T_{\max})$ by the above lemma, we get $H(u(t)) \geq \mu$ by the above lemma. Thus, it follows that
\begin{align*}
	I''(t) + I'(t) 
	\geq \l( 1+ \frac{d}{d-2} \r) \norm{\partial_t u(t)}_{L^2}^2 + \frac{2d}{d-2} \l( \mu - E(u(t), \partial_t u(t))\r). 
\end{align*}
\end{proof}

\begin{lemma}
Assume that $u$ is a solution to \eqref{NLDW} on $[0,\infty)$ with $(u_0,u_1) \in \scB$. Then, there exists $t_1>0$ such that $I(t)>0$, $I'(t)>0$, and 
\begin{align*}
	\frac{d}{dt} \l( \frac{E(u(t),\partial_t u(t)) - \mu }{I^{\frac{d-1}{d-2}}(t)}\r) \leq 0
\end{align*}
for all $t \in (t_1,\infty)$. 
\end{lemma}

\begin{proof}
We set $E(t):=E(u(t),\partial_t u(t))$ for simplicity. We define
\begin{align*}
	F(t):= I'(t) + \l(1+ \frac{d}{d-2}\r) \l( E(t) - \mu \r).
\end{align*}
It is follows from the energy and the lemma that
\begin{align*}
	F'(t)
	&=I''(t) +  \l(1+ \frac{d}{d-2}\r)  E'(t) 
	\\
	&=I''(t) -  \l(1+ \frac{d}{d-2}\r)  \norm{\partial_t u(t)}_{L^2}^2
	\\
	&\geq -I'(t) + \frac{2d}{d-2} \l( \mu - E(t)\r)
	\\
	&= -F(t)+ \frac{2}{d-2} \l( \mu - E(t)\r)
\end{align*}
Therefore, we have
\begin{align*}
	(e^{t} F(t))' 
	&= e^{t} F(t) + e^{t} F'(t) 
	\\
	&\geq e^{t} \frac{2}{d-2} \l( \mu - E(t)\r)
	\\
	& \geq  \frac{2}{d-2} \l( \mu - E(0)\r) e^{t}
\end{align*}
Integrating this on $[0,t]$, we obtain
\begin{align*}
	 F(t) \geq  F(0) e^{-t} + \frac{2}{d-2} \l( \mu - E(0)\r)  -  \frac{2}{d-2} \l( \mu - E(0)\r) e^{-t}.
\end{align*}
Thus there exists $t_0 >0$ such that 
\begin{align*}
	F(t) >0 \text{ for all } t \geq t_0,
\end{align*}
since $ \mu > E(u_0, u_1)$. This implies that, for all $t \geq t_0$,  
\begin{align}
	\label{eq6.2}
	I'(t) 
	&\geq - \l(1+ \frac{d}{d-2}\r) \l( E(t) - \mu \r) 
	\\ \notag
	&=  \l(1+ \frac{d}{d-2}\r) \l( \mu - E(t)\r) 
	\\ \notag
	&\geq \l(1+ \frac{d}{d-2}\r) \l( \mu - E(0)\r) 
	\\ \notag
	&>0.
\end{align}
Therefore, there exists $t_1 \geq  t_0$ such that 
\begin{align*}
	I(t)>0 \text{ for all } t \geq t_1. 
\end{align*}
For $t \geq t_1$, we have
\begin{align}
	\label{eq6.3}
	\frac{d}{dt} \l( \frac{E(t) - \mu }{I^{\frac{d-1}{d-2}}(t)}\r)
	= \frac{E'(t) I(t) - \frac{d-1}{d-2} I'(t) (E(t) - \mu )}{I^{\frac{2d-3}{d-2}}(t)}.
\end{align}
By \eqref{eq6.1} and \eqref{eq6.2}, we obtain
\begin{align*}
	&E'(t) I(t) - \frac{d-1}{d-2} I'(t) (E(t) - \mu )
	\\
	&\leq - \norm{\partial_t u(t)}_{L^2}^2 I(t) - \frac{d-1}{d-2} I'(t) (E(t) - \mu )
	\\
	&\leq -2 \norm{\partial_t u(t)}_{L^2}^2 \norm{u(t)}_{L^2}^2  + \frac{d-1}{d-2} \l(1+ \frac{d}{d-2}\r) (I'(t))^{2} 
	\\
	&=  -2 \norm{\partial_t u(t)}_{L^2}^2 \norm{u(t)}_{L^2}^2  +2 (I'(t))^{2} 
	\\
	&\leq 0, 
\end{align*}
where we used the Cauchy--Schwarz inequality $I'(t) \leq \| u(t) \|_{L^2} \| \partial_t u(t) \|_{L^2}$ in the last inequality. This and \eqref{eq6.3} implies the statement. 
\end{proof}

\begin{proof}[Proof of Theorem \ref{thm1.5}]
We suppose that the solution $u$ to \eqref{NLDW} with $(u_0,u_1) \in \scB$ exists globally in time. Then, we have
\begin{align*}
	I(t)>0, 
	\quad I'(t)>0,
	\quad \mu- E(t) \geq  C I^{\frac{d-1}{d-2}}(t)
\end{align*}
for large $t\geq t_1$, where we set $C=(\mu-E(t_1))/I^{\frac{d-1}{d-2}}(t_1)>0$ and we recall $E(t)=E(u(t),\partial_t u(t))$. By Lemma \ref{lem5.5}, we get
\begin{align*}
	I''(t) + I'(t) 
	&\geq  \l(1+\frac{d}{d-2}\r) \norm{\partial_t u(t)}_{L^2}^2 + \frac{2d}{d-2} \l(\mu - E(t)\r)
	\\
	&\geq  \frac{2d}{d-2} \l(\mu - E(t)\r)
	\\
	&\geq C  I^{\frac{d-1}{d-2}}(t),
\end{align*}
for $t \geq t_1$. We also have $I(t_1)>0$ and $I'(t_1)>0$. Thus, by Lemma \ref{lem5.1}, $I(t)$ can not exist globally. This contradicts the assumption that the solution $u$ is global. Therefore, we get the statement. 
\end{proof}

%%%%%%%%%%%%%%%%%%%%%%%%%%%%%%%%%%%%%%%%%%%%%%%%%%%%%%%%%%%%%%%%

\begin{acknowledgement}
The author would like to express deep appreciation to Professor Masahito Ohta and Professor Yuta Wakasugi for many useful suggestions, valuable comments and warm-hearted encouragement. The author was partially supported by JSPS Grant-in-Aid for Early-Career Scientists JP18K13444.
\end{acknowledgement}

\end{document}